\documentclass{amsart}[12pt]
\usepackage[utf8]{inputenc}
\usepackage{amsmath,amssymb,amstext,amsthm, geometry, verbatim, tikz-cd, enumerate}
\usepackage[all]{xy}
\newtheorem{theorem}{Theorem}[section]
\newtheorem{lemma}[theorem]{Lemma}
\newtheorem{proposition}[theorem]{Proposition}

\theoremstyle{definition}
\newtheorem{definition}[theorem]{Definition}
\newtheorem{example}{Example}[section]
\theoremstyle{remark}
\newtheorem{remark}[theorem]{Remark}
\numberwithin{equation}{section}

\newcommand \Z{\mathbb Z}
\newcommand \R{\mathbb R}

\newcommand \C{\mathbb C}

\newcommand \id{\mathrm{Id}}

\newcommand \mc{\mathcal}
\newcommand \mb{\mathbb}

\newcommand \sHom{\mc{H}om}
\newcommand \sEnd{\mc{E}nd}
\newcommand \sExt{\mc{E}xt}
\newcommand \Div{[\mathrm{Div}(X)]}
\DeclareMathOperator{\Hom}{Hom}

\DeclareMathOperator{\Pic}{Pic}
\DeclareMathOperator{\GL}{GL}

\DeclareMathOperator{\rank}{rk}
\DeclareMathOperator{\Span}{span}

\DeclareMathOperator{\im}{Im}

\DeclareMathOperator{\Ext}{Ext}

\DeclareMathOperator{\tr}{Tr}
\DeclareMathOperator{\coker}{coker}
\DeclareMathOperator{\chern}{ch}

\title{{P}oisson three-folds constructed from co-{H}iggs bundles on {H}opf surfaces}
\author{Eric Boulter}
\address{Department of Mathematics and Statistics, University of Saskatchewan, SK, Canada~ S7N 5E6}
\email{eric.boulter@usask.ca, erictboulter@gmail.com}
\subjclass[2020]{Primary: 32L05, Secondary: 53D17, 14D21}
\date{\today}
\begin{document}
\begin{abstract}
    This paper extends a previous work in which the rank-2 co-Higgs bundles on a Hopf surface are classified based on the data of the underlying vector bundle. The aim of the paper is to study the Poisson 3-folds that can be constructed from these co-Higgs bundles by describing their symplectic leaves.
\end{abstract}
\maketitle
\section{Introduction}

Poisson structures are geometric objects arising from the ideas of Hamiltonian mechanics where, using a Poisson structure on the phase space, one can recover evolution equations for a system from the conserved Hamiltonian, a function typically representing the total energy of the system. A Poisson structure on a manifold also naturally gives rise to an even-dimensional foliation of the manifold into symplectic leaves where the restriction of the Poisson structure is invertible. The dynamics given by a Hamiltonian on this Poisson manifold are constrained to the symplectic leaves, so understanding this foliation gives insight into conserved quantities of such dynamics. In this paper, we will investigate these symplectic foliations on a particular family of Poisson manifolds: $\mb{P}^1$-bundles $p:Y\to X$ over a compact complex manifold $X$ whose fibres $p^{-1}(x)$ are co-isotropic. In other words, we look at Poisson structures on $Y$ such that for any function $f \in \mc{O}_X$, the Hamiltonian flow with respect to $p^*(f)$ preserves fibres of $p$. The Poisson structures of this type can be studied in terms of rank-2 trace-free co-Higgs bundles on the base manifold $X$ \cite{Polishchuk97}. This problem has been studied in detail for the case when $X$ is curve \cite{Bartocci-Macri05}, leaving the case where $X$ is a surface as the next natural setting for this problem. Based on \cite{Rayan-14}, we expect the more interesting cases to occur for $X$ of Kodaira dimension $-\infty$. In this paper, we will focus on the case where $X$ is a Hopf surface as these surfaces have some advantages in computational tractability over other Kodaira dimension $-\infty$ surfaces.

In Section \ref{Poisson-co-Higgs} we discuss this correspondence between co-Higgs bundles and associated Poisson structures in more detail, and describe a canonical factoring for a rank-2 trace-free co-Higgs bundle which will be used for later results.

In Sections \ref{Hopf-classification} and \ref{Hopf-flat} we recall the classification of primary Hopf surfaces and the structure of flat bundles of rank 1 and 2  on each class of Hopf surface following \cite{Mall91,Mall93}. We provide additional details both to aid in later computations and to correct some minor errors from the original referenced articles.

In Section \ref{flat-co-Higgs}, we directly compute all possible flat rank-2 trace-free co-Higgs bundles for each type of primary Hopf surface using our factorization from Section \ref{Poisson-co-Higgs}. In this section we also show that any non-zero rank-2 trace-free co-Higgs bundle on a primary Hopf surface will have an underlying vector bundle which is filtrable, meaning that it admits a sub-line bundle (which may or may not be invariant with respect to the Higgs field).

In Section \ref{co-Higgs-non-flat}, we consider co-Higgs bundles with non-zero second Chern class. The method for constructing these co-Higgs bundles varies depending on the algebraic dimension of the Hopf surface. On a Hopf surface with algebraic dimension 1, every filtrable bundle with non-zero second Chern class can be uniquely described via a sequence of elementary modifications ending in a flat bundle. We describe the effect of elementary modifications on co-Higgs bundles to reduce to the case of Section \ref{flat-co-Higgs}. For the case of Hopf surfaces with algebraic dimension 0, not every filtrable bundle can be constructed from a flat bundle via elementary modifications, but we give another method to describe Higgs fields for the bundles which cannot be constructed in this way.

In Section \ref{foliations}, we give some results that help to describe the symplectic foliations for a Poisson structure arising from a rank-2 trace-free co-Higgs bundle and look at some example computations relating to co-Higgs bundles constructed in the previous sections.

\section{Co-Higgs bundles and Poisson structures}\label{Poisson-co-Higgs}

\subsection{Complex Poisson manifolds}
A \emph{Poisson structure} on a complex manifold $X$ consists of a Lie bracket $\{\cdot,\cdot\}:\mc{O}_X\times\mc{O}_X\to \mc{O}_X$ satisfying the Leibniz rule $\{fg,h\}=f\{g,h\}+g\{f,h\}$ for any open set $U\subseteq X$ and functions $f,g,h\in \mc{O}_X(U)$. 
A Poisson structure can also be represented as a bivector $\pi \in H^0(X,\wedge^2\mc{T}_X)$ so that $\pi(\mathrm{d}f,\mathrm{d}g)=\{f,g\}.$ 
Conversely, a bivector $\sigma$ corresponds to a Poisson structure if and only if the Schouten bracket $[\sigma,\sigma]\in H^0(X,\wedge^3\mc{T}_X)$ vanishes.

Given a Poisson bivector $\pi$, we can associate to any holomorphic function $f$ a \emph{Hamiltonian vector field} $H_f:=\iota_{\mathrm{d}f}\pi$. At any point $x$, the flows of the Hamiltonian vector fields $H_f$ for all functions $f$ defined near $x$ generate an immersed complex subvariety $S$ so that $\pi|_S$ is invertible when viewed as a map from $\mc{T}^*_X$ to $\mc{T}_X$. 
We call $S$ a \emph{symplectic leaf} of $\pi$, since the inverse of $\pi|_S$ is a holomorphic symplectic form on $S$. 
These symplectic leaves give a foliation of $X$, called the \emph{symplectic foliation} associated to $\pi$.

\subsection{Co-Higgs bundles}
Let $X$ be a compact complex manifold. 
A \emph{co-Higgs bundle} on $X$ consists of a pair $(E,\Phi)$, where $E$ is a holomorphic vector bundle on $X$, and $\Phi$ is a holomorphic section of $\sEnd(E)\otimes \mc{T}_X$ called the \emph{Higgs field} so that $\Phi\wedge \Phi=0$ where the wedge product acts as the commutator on the $\sEnd(E)$ factor and as the vector field wedge product on the $\mc{T}_X$ factor. 
Using the decomposition $\sEnd(E)=\mc{O}_X\oplus\sEnd_0(E)$ into the identity and trace-free components, we can write any Higgs field $\Phi$ as $\mathrm{Id}_E\otimes \xi+\Phi_0$, where $\xi$ is a holomorphic vector field and $\Phi_0 \in H^0(X,\sEnd_0(E)\otimes \mc{T}_X)$. 
If $\xi=0$, we say that the co-Higgs bundle is \emph{trace-free}. 
As shown in \cite{Polishchuk97} and \cite{RayanThesis}, to any rank-2 trace-free Higgs bundle $(E,\Phi)$ we can associate a Poisson manifold $(\mb{P}(E),\sigma_\Phi)$, where the Poisson structure $\sigma_\Phi$ is constructed from the bivector $\sigma_\Phi:=\Phi\in H^0(X,\sEnd_0(E)\otimes \mc{T}_X)\simeq H^0(\mb{P}(E),\mc{T}_p\otimes p^*\mc{T}_X)\subset H^0(\mb{P}(E),\wedge^2\mc{T}_{\mb{P}(E)})$, where $p:\mb{P}(E)\to X$ is the bundle map and $\mc{T}_p$ is the relative tangent sheaf. 
This comes by considering the natural exact sequence 
$$\begin{tikzcd}
    0  \arrow[r] & \mc{T}_p \arrow[r] &\mc{T}_{\mb{P}(E)} \arrow[r] & p^*\mc{T}_X \arrow[r] & 0,
\end{tikzcd}$$
and noting that $p_*\mc{T}_p\simeq \sEnd_0(E)$ since fibres of $\sEnd_0(E)$ parameterize deformations of the projective line fibres of $\mb{P}(E)$. 
The condition that the Poisson structure satisfies the Jacobi identity is equivalent to the condition that $\Phi\wedge\Phi=0$ \cite[p.1425, Ex 1]{Polishchuk97}, \cite[Exercise 2.12]{Pym18}.

The Poisson structures on $\mb{P}(E)$ arising in this way are precisely those for which the fibres of the projective bundle structure are co-isotropic, meaning that for any fibre $F=p^{-1}(x), x \in X$, we have that $\pi(I_F/I_F^2)\subseteq \mc{T}_F$ when we interpret $\pi$ as a bundle map from $\mc{T}_{\mb{P}(E)}^*$ to $\mc{T}_{\mb{P}(E)}$.

In the rank-2 trace-free case specifically, we can describe the co-Higgs bundles concretely in terms of line bundle-twisted Higgs bundles. 
\begin{lemma}\label{factor-Higgs}
    Let $X$ be a compact complex manifold, and let $(E,\Phi)$ be a rank-2 trace-free co-Higgs bundle on $X$. Then there is a line bundle $\lambda\in \Pic(X)$, a $\lambda$-twisted Higgs field $\phi_0 \in H^0(X,\sEnd_0(E)\otimes \lambda)$, and a meromorphic vector field $\xi \in H^0(X,\mc{T}_X\otimes \lambda^{-1})$ such that $\Phi=\phi_0\otimes \xi$. If $\phi_0$ is assumed to vanish only on a set of codimension 2, then $\lambda$ is unique, and $\Phi$ is determined by $\phi_0$ and $\xi$ up to the $\C^*$-action $t\cdot (\phi_0\otimes \xi)=(t\phi)\otimes (t^{-1}\xi)$.
\end{lemma}

\begin{proof}
    Let $U$ be an open subset trivializing both $E$ and $\mc{T}_X$, and let $\xi_1, \ldots \xi_n$ be a local frame for $\mc{T}_X$ defined on $U$. We can write $\Phi|_U= \sum\limits_{i=1}^n \psi_i\otimes\xi_i$, where the $\psi_i$ are trace-free matrices with coefficients in $\mc{M}_X(U)$. Since $\Phi$ is a co-Higgs bundle, we must have $\Phi\wedge\Phi|_U=\sum\limits_{i=1}^{n-1}\sum\limits_{j=i+1}^n[\psi_i,\psi_j]\otimes \xi_i\wedge \xi_j=0.$ Since $\xi_i\wedge \xi_j$ is nowhere-vanishing on $U$, this implies that $[\psi_i,\psi_j]=0$, which can only occur if there are functions $f_i,f_j\in \mc{O}_X(U)$ such that $f_j\psi_1=f_i\psi_j$. From this we conclude that $\Phi|_U=\psi_0\otimes(\sum\limits_{i=1}^n f_i\xi_i)$ for some $\mc{M}_X(U)$-valued matrix $\psi_0$.

    Now let $U$ and $U'$ be two open sets which trivialize $E$ and $\mc{T}_X$. Then $\Phi|_U=\psi\otimes \xi$ and $\Phi|_{U'}=\psi'\otimes \xi'$ with $(\psi\otimes \xi)|_{U'}=(\psi'\otimes \xi')|_{U}$. This is only possible if there is a non-vanishing function $f\in \mc{O}_X^*(U\cap U')$ so that $\psi=f\psi'$ and $f\xi=\xi'$, so the pure tensor description of $\Phi$ extends to $U\cup U'$ up to twisting factors by a line bundle. Applying this argument on a suitable open cover gives the result.

    Suppose that there is another collection $(\lambda', \phi_0',\xi')$ with $\lambda'\in \Pic(X)$, $\phi_0'\in H^0(X,\sEnd_0(E)\otimes \lambda')$, and $\xi'\in H^0(X,\mc{T}_X\otimes (\lambda')^{-1})$ so that $\Phi=\phi_0'\otimes \xi'.$ Then there is a meromorphic function $f$ so that $\phi_0'=f\phi_0$ and $\xi=f\xi'.$ If we assume $\phi_0$ does not vanish on any codimension-1 subset, then either $\phi_0'$ vanishes on the zeros of $f$, or $f$ is a constant and $\lambda'\simeq\lambda$.
\end{proof}

We will make use of this restricted structure to describe symplectic foliations in terms of the meromorphic vector field $\xi$ associated to a co-Higgs bundle. In the sequel, all twisted Higgs bundles $(E,\Phi)$ will be assumed to have $\rank E=2$ and $\tr\Phi=0$ except where otherwise stated.

\section{Hopf surfaces}\label{Hopf-classification}
A (primary) Hopf surface is a compact complex surface homeomorphic to $S^1\times S^3$. A general Hopf surface is described by four parameters $a,b,\lambda \in \C, r\in \Z_{> 0}$ with $0<|a|\leq |b|<1$ and $(a-b^r)\lambda=0$ as the quotient of $W=\mathbb{C}^2\setminus\{0\}$ by the group of automorphisms generated by $\varphi:(x,y)\mapsto (ax+\lambda y^r,by)$. We can separate Hopf surfaces into five groups, each with different characteristic behaviour \cite{Mall91}:
\begin{enumerate}
    \item (Classical) $\lambda=0$ and $a=b$. 
    \item (Resonant) $\lambda=0$, $r>1$, and $a=b^r$. 
    \item (Hyper-resonant) $\lambda=0$, and there are integers $r_1>1,r_2>1$ so that $a^{r_1}=b^{r_2}$ but $a^{k_1}\neq b^{k_2}$ for integers $k_1,k_2$ when $1\leq k_1<r_1$, $1\leq k_2<r_2$.
    \item (Generic) $\lambda=0$, $a^{k_1}\neq b^{k_2}$ for any positive integers $k_1,k_2$.
    \item (Exceptional) $\lambda\neq 0$ and $a=b^r$.
\end{enumerate}
We say that a Hopf surface is diagonal if it has $\lambda=0$, in which case it belongs to one of the families (1)-(4).
\begin{remark}
    If $X_{(a,b,\lambda,r)},X_{(a',b',\lambda',r')}$ are two Hopf surfaces defined by parameters $(a,b,\lambda,r)$ and $(a',b',\lambda',r')$, respectively, then $X\simeq X'$ if and only if there is an isomorphism $\sigma:W\to W$ so that $\varphi\circ\sigma=\sigma\circ\varphi'$, where $\varphi=(x,y)\mapsto (ax+\lambda y^r, by)$ and $\varphi'=(x,y)\mapsto (a'x+\lambda'y^{r'},b'y)$.
    
    Using this fact, we can easily check that two diagonal Hopf surfaces $X_{(a,b,0,r)}$ and $X_{(a',b',0,r')}$ are isomorphic if and only if $\{a,b\}=\{a',b'\}$. We also get that the Hopf surface $X_{(b^r,b,\lambda,r)}$ is isomorphic to $X_{(b^r,b,1,r)}$ whenever $\lambda\neq 0$ using the map $\sigma:(x,y)\mapsto (x,ty)$ for some $t$ satisfying $t^r=\lambda$.
\end{remark}

\begin{definition}
    Let $X$ be a complex manifold with fundamental group $\Gamma$ and universal cover $\gamma:Y\to X$. An \emph{$r$-dimensional factor of automorphy} for $X$ is a holomorphic map $A:\Gamma\times Y \to \mathrm{GL}(r,\mc{O}_Y)$ satisfying $A(\lambda\mu, y)=A(\lambda, \mu\cdot y)A(\mu,y)$ for any $\lambda,\mu \in \Gamma$ and $y \in Y$, where $\mu\cdot y$ is the action of $\mu$ by the corresponding deck transformation. The vector bundle $E_A$ represented by a factor of automorphy $A$ is the quotient of the trivial vector bundle $Y\times \C^r$ by the equivalence relation $(y,v) \sim (\lambda\cdot y, A(\lambda,y)v)$ for all $\lambda \in \Gamma$.
\end{definition}
\begin{remark}
    For a Hopf surface $X$, which has universal cover $W=\C^2\setminus\{0\}$ and the fundamental group $\Z$, any factor of automorphy $A:\Z\times W\to \mathrm{GL}(2,\mc{O}_W)$ is completely determined by $A|_{\{1\}\times W}$. In the sequel we will write factors of automorphy for $X$ as sections of $\mathrm{GL}(r,\mc{O}_W).$
\end{remark}
On any Hopf surface $X$, $\Pic(X)\simeq \C^*$ as every line bundle is isomorphic to one of the form $L_t$, where $L_t$ is the line bundle represented by a constant factor of automorphy $t$ for some $t\in \C^*\subset \mathrm{GL}(1,\mc{O}_W)$. Holomorphic sections of a bundle $L_t$ are in one-to-one correspondence with holomorphic sections $s \in \mc{O}_W$ with $s(\varphi(x,y))=ts(x,y)$. In particular, for a diagonal Hopf surface $H^0(X,L_t)=\Span\{x^ky^\ell:k,\ell\in \Z_{\geq 0}, a^kb^\ell=t\}$ and for an exceptional Hopf surface $H^0(X,L_t)=\Span\{y^k:k \in \Z_{\geq 0}, b^k=t\}$. Every primary Hopf surface $X$ has $K_X\simeq L_{ab}$ and every line bundle $L_t \in \Pic(X)$ has $\chi(L_t)=0$, so all cohomology groups of line bundles can be calculated using the correspondence for global sections. If we let $\Div$ be the subgroup of $\Pic(X)$ generated by line bundles with a non-zero section, we can describe the bundles in $\Div$ as in Table \ref{line-bundles-Hopf} \cite{Mall91}.
\begin{table}[ht]
    \centering
    \begin{tabular}{|c|c|c|c|}
    \hline
 Surface type & $\Div$ & Generators & Cohomology\\
 \hline
 Classical & $\Z$ & $L_b$ & $h^0(X,(L_b)^k)=\max\{0,k+1\}$\\
 & & & $h^1(X,(L_b)^k)=|1+k|$\\
 & & & $h^2(X,(L_b)^k)=\max\{0,-k-1\}$\\
 \hline
 Resonant & $\Z$ & $L_b$ & $h^0(X,(L_b)^k)=\max\left\{0,\left\lfloor\frac{k}{r}\right\rfloor+1\right\}$\\
 & & & $h^1(X,(L_b)^k)=\left|1+\left\lfloor\frac{k}{r}\right\rfloor\right|$\\
 & & & $h^2(X,(L_b)^k)=\max\left\{0,-\left\lfloor\frac{k}{r}\right\rfloor-1\right\}$\\
 \hline
 Hyper-resonant & $\Z\oplus \Z/d\Z$ & $L_a$, $L_b$ & $h^0(X,(L_a)^k\otimes (L_b)^\ell)=\max\left\{0,\left\lfloor \frac{k}{r_1}\right\rfloor+\left\lfloor \frac{\ell}{r_2}\right\rfloor+1\right\}$\\
 & $d=\gcd(r_1,r_2)$ & & $h^1(X,(L_a)^k\otimes (L_b)^\ell)=\left|1+\left\lfloor \frac{k}{r_1}\right\rfloor+\left\lfloor \frac{\ell}{r_2}\right\rfloor\right|$\\
 & & & $h^2(X,(L_a)^k\otimes (L_b)^\ell)=\max\left\{0,-\left\lfloor \frac{k}{r_1}\right\rfloor-\left\lfloor \frac{\ell}{r_2}\right\rfloor-1\right\}$\\
 \hline
 Generic & $\Z\oplus \Z$ & $L_a$,$L_b$ & $h^0(X,(L_a)^k\otimes (L_b)^\ell)=\begin{cases}
     1 & k\geq 0, \ell\geq 0,\\
     0 & \text{otherwise}.
 \end{cases}$\\
 & & & $h^1(X,(L_a)^k\otimes (L_b)^\ell)=\begin{cases}
     1 & k\geq 0, \ell \geq 0 \text{ or } k\leq -1, \ell\leq -1\\
     0 & \text{otherwise}.
 \end{cases}$\\
 & & & $h^2(X,(L_a)^k\otimes (L_b)^\ell)=\begin{cases}
     1 & k\leq -1, \ell\leq -1\\
     0 & \text{otherwise}.
 \end{cases}$\\
 \hline
 Exceptional & $\Z$ & $L_b$ & $h^0(X,(L_b)^k)=\begin{cases}
     1 & k\geq 0,\\
     0 & \text{otherwise.}
 \end{cases}$\\
 & & & $h^1(X,(L_b)^k)=\begin{cases}
     1 & k\geq 0 \text{ or } k\leq -r-1,\\
     0 & \text{otherwise.}
 \end{cases}$\\
 & & & $h^2(X,(L_b)^k)=\begin{cases}
     1 & k\leq -r-1,\\
     0 & \text{otherwise.}
 \end{cases}$\\
 \hline
\end{tabular}
    \caption{The subgroup $\Div$ for all primary Hopf surfaces}
    \label{line-bundles-Hopf}
\end{table}

For any Hopf surface, there is a degree function $\deg:\Pic(X)\to \R$ given by $\deg(L_t):=\frac{\log{|t|}}{\log{|b|}}.$ Since Hopf surfaces have a connected Picard group, this degree function is unique up to multiplication by a positive constant.
\begin{definition}
    Let $\mc{E}$ be a torsion-free sheaf of rank $r>0$. The \emph{slope} of $\mc{E}$ is given by $\mu(\mc{E}):=\frac{\deg(\det(\mc{E}))}{r}.$ We say that $\mc{E}$ is $\mu$-(semi-)stable if for any subsheaf $\mc{F}\subset \mc{E}$ with $\rank(\mc{F})<\rank(\mc{E})$, $\mu(\mc{F})< (\leq) \mu(E)$.
\end{definition}

For diagonal Hopf surfaces, the tangent bundle is isomorphic to the direct sum $T_X\simeq L_a\oplus L_b$, but the tangent bundle of an exceptional Hopf surface is indecomposable \cite{Beauville00}. In the exceptional case the tangent bundle can be described as an extension of line bundles using the computations in Section \ref{exceptional-extension}.

\section{Extensions of line bundles}\label{Hopf-flat}
Much of this section follows \cite{Mall93}, though we give additional detail for some cases, particularly the exceptional case where some formulae are misstated.

We start by noting the following standard fact to simplify future computations.
\begin{lemma}
    Let $X$ be any compact complex manifold, and let $L \in \Pic(X)$ be a line bundle with $H^1(X,L)\neq 0$. If $X$ admits a Gauduchon metric $g$ with $\deg_g(L)\geq 0$, then any non-split extension $$\begin{tikzcd}
        0 \arrow[r] & L \arrow{r}{\iota} & E \arrow{r}{\rho} & \mc{O}_X \arrow[r] & 0
    \end{tikzcd}$$ is indecomposable.
\end{lemma}
\begin{proof}
    Note that any decomposition $E\simeq \lambda\oplus \lambda'$ for some $\lambda, \lambda'\in \Pic(X)$ will have $\deg_g(\lambda)+\deg_g(\lambda')=\deg_g(L),$ so without loss of generality any such decomposition will have $\deg_g(\lambda)\geq \frac{1}{2}\deg_g(L)\geq 0$.
    
    Let $\lambda \in \Pic(X)$ be such that there is a non-zero section $s \in H^0(X,E\otimes \lambda^{-1})$. We show that either $\rho\circ s=0$ or $\deg_g(\lambda)<0$. Suppose that $\rho\circ s\neq 0$. Then $H^0(X,\lambda^{-1})\neq 0$, so either $\lambda \simeq \mc{O}_X$ or $\deg_g(\lambda)<0$. However, if $\lambda \simeq \mc{O}_X$, then $s$ would be a section of $\rho$, contradicting the fact that $\iota$ and $\rho$ form a non-split exact sequence, so $\deg_g(\lambda)<0$.

    Finally, a section $s \in H^0(X,E\otimes \lambda^{-1})$ with $\rho\circ s=0$ cannot generate a decomposition of $E$, since either $\lambda\simeq L$ or $\coker(s)$ has non-zero torsion. 
    \end{proof}

    \subsection{Extensions of line bundles for diagonal Hopf surfaces}
    If $X$ is a Hopf surface given by the map $\varphi:(x,y)\mapsto (ax,by)$ for some $a,b \in \C^*$ with $|a|\leq |b|<1$, there are two main paradigms for non-split extensions $$\begin{tikzcd}
        0 \arrow[r] & \lambda_1 \arrow[r] & E \arrow[r] & \lambda_2 \arrow[r] & 0
    \end{tikzcd}$$
    for $\lambda_1,\lambda_2 \in \Pic(X)$, depending on whether $\deg(\lambda_1)\geq \deg(\lambda_2)$ or $\deg(\lambda_1)< \deg(\lambda_2)$. In the first case, the extensions can be classified as follows:
    \begin{proposition}\label{diag-indecomposable}
        Take $X$ a diagonal Hopf surface, and let $\lambda_1,\lambda_2 \in \Pic(X)$ be such that $H^1(X,\lambda_1\otimes \lambda_2^{-1})\neq 0$ and $\deg(\lambda_1)\geq\deg(\lambda_2)$. In this case $H^1(X,\lambda_1\otimes\lambda_2^{-1})\simeq H^0(X,\lambda_1\otimes \lambda_2^{-1})$. Up to isomorphism, the non-split extensions $$\begin{tikzcd}
            0\arrow[r] & \lambda_1\arrow[r] & E \arrow[r] & \lambda_2\arrow[r] & 0
        \end{tikzcd}$$
        are parameterized by $\mb{P}(H^0(X,\lambda_1\otimes \lambda_2^{-1})$, with the extension corresponding to $f \in H^0(X,\lambda_1\otimes\lambda_2^{-1})\setminus \{0\}$ given by the factor of automorphy $$A=\begin{pmatrix}
            m_1 & f\\0 & m_2
        \end{pmatrix},$$ where $m_1,m_2\in \C^*$ are the factors of automorphy corresponding to $\lambda_1$ and $\lambda_2$, respectively. For any $L_\tau \in \Pic(X)$, $H^0(X,E\otimes L_\tau)\simeq H^0(X,\lambda_1\otimes L_\tau).$ Note that here we are identifying $f$ with its lift to $W$ via the constant factor of automorphy for $\lambda_1\otimes \lambda_2$ so that $A$ is a well-defined factor of automorphy.
        \end{proposition}
        \begin{proof}
            Clearly the bundle $E_A$ is an extension of $\lambda_2$ by $\lambda_1$, so to show that it is not split, it will suffice to check that $h^0(X,E_A\otimes \lambda_2^{-1})<h^0(X,\lambda_1\otimes \lambda_2^{-1})+h^0(X,\mc{O}_X)=h^0(X,\lambda_1\otimes \lambda_2^{-1})+1.$ We will show this by checking the stronger statement $h^0(X,E_A\otimes L_\tau)=h^0(X,\lambda_1\otimes L_\tau)$ for any $\tau \in \C^*$. Suppose that the pair $\alpha,\beta \in \mc{O}_W$ give a section of $E\otimes L_\tau$. Then $$\begin{pmatrix}
                \alpha(ax,by)\\\beta(ax,by)
            \end{pmatrix}=\tau A\begin{pmatrix}
                \alpha(x,y)\\\beta(x,y)
            \end{pmatrix}=\begin{pmatrix}
                m_1\tau\alpha(x,y)+\tau f(x,y)\beta(x,y)\\ m_2\tau \beta(x,y)
            \end{pmatrix}.
            $$
            Clearly $\beta \in H^0(X,\lambda_2\otimes L_\tau)$, and $f\in H^0(X,\lambda_1\otimes \lambda_2^{-1})$ by assumption, so there is a $g \in H^0(X,\lambda_1\otimes L_\tau)$ with $g(x,y)=f(x,y)\beta(x,y)$.
            Writing out the Taylor series $\alpha(x,y)=\sum\limits_{i,j}\alpha_{i,j}x^iy^j$ and $g(x,y)=\sum\limits_{i,j}g_{i,j}x^iy^j$ and noting that $g_{i,j}=0$ if $a^ib^j\neq m_1\tau$, the first equation for $(\alpha,\beta)$ becomes \begin{align*}
            0&=\tau g_{i,j} && i,j \in \Z_{\geq 0}, a^ib^j=m_1\tau,\\
            (a^ib^j-m_1\tau)\alpha_{i,j}&=0 && i,j \in \Z_{\geq 0}, a^ib^j\neq m_1\tau,
        \end{align*}
        so that $g_{i,j}=0$ for all $(i,j) \in \Z_{\geq 0}^2$. Since $\mc{O}_W$ is an integral domain, $g=0$ and $f\neq 0$ imply that $\beta=0$. Making this substitution shows that $(\alpha,\beta)$ give a section of $E\otimes L_\tau$ if and only if $\beta=0$ and $\alpha \in H^0(X,\lambda_1\otimes L_\tau)$, so $h^0(X,E\otimes L_\tau)=h^0(X,\lambda_1\otimes L_\tau)$ for any $\tau \in \C^*$. Therefore $E_A$ is a non-split extension.
        
        Suppose that $$A'=\begin{pmatrix}
            m_1 & f'\\0 & m_2
        \end{pmatrix}$$ is another factor of automorphy with $f'\in H^0(X,\lambda_1\otimes \lambda_2^{-1})\setminus\{0\}$. Then $E_A\simeq E_{A'}$ if and only if there is a matrix $B \in \GL(2,\mc{O}_W)$ so that $B(ax,by)E_A=E_{A'}B(x,y)$. Writing $B=\begin{pmatrix}
            s & t\\u & v
        \end{pmatrix}$ and expanding this equation gives $$\begin{pmatrix}
            m_1s(ax,by) & f(x,y)s(ax,by)+m_2t(ax,by)\\m_1 u(ax,by) & f(x,y)u(ax,by)+m_2v(ax,by)
        \end{pmatrix}=\begin{pmatrix}m_1s(x,y)+f'(x,y)u(x,y) & m_1t(x,y)+f'(x,y)v(x,y)\\m_2u(x,y) & m_2v(x,y)\end{pmatrix}.$$
        The bottom-left entry gives $u(ax,by)=\frac{m_2}{m_1}u(x,y)$, so $u \in H^0(X,\lambda_2\otimes \lambda_1^{-1})$. Setting $g(x,y)=f'(x,y)u(x,y)$ and writing $s(x,y)=\sum\limits_{i,j}s_{i,j}x^iy^j$ and $g(x,y)=\sum\limits_{i,j}g_{i,j}x^iy^j$, the top-left entry gives $$\sum\limits_{i,j}m_1a^ib^js_{i,j}x^iy^j=\sum\limits_{i,j}(m_1s_{i,j}+g_{i,j})x^iy^j.$$ By our assumption on $f'$, $g_{i,j}=0$ whenever $a^ib^j \neq 1$, so comparing coefficients gives \begin{align*}
            m_1a^ib^js_{i,j}&=m_1s_{i,j} && a^ib^j\neq 1,\\
            0&=g_{i,j} && a^ib^j=1.
        \end{align*}
        We can conclude from this that $s(x,y)=s_0$ for some $s_0 \in \C$ and $g(x,y)=u(x,y)=0$.

        Since $u=0$, the bottom-right entry only holds if $v(x,y)=v_0$ for some $v_0\in \C$. Setting $t(x,y)=\sum\limits_{i,j}t_{i,j}x^iy^j$ and $s_0f(x,y)-v_0f'(x,y)=\sum\limits_{i,j}f_{i,j}x^iy^j$ the top-right entry becomes $$\sum\limits_{i,j}(a^ib^jm_2t_{i,j}+f_{i,j})x^iy^j=\sum\limits_{i,j}m_1t_{i,j}x^iy^j.$$ By the assumption that $f,f'\in H^0(X,\lambda_1\otimes \lambda_2^{-1})$, $f_{i,j}=0$ when $a^ib^j \neq \frac{m_1}{m_2}$, so coefficient-wise equations for the top-right entry are \begin{align*}
            f_{i,j}&=0 & a^ib^j =\frac{m_1}{m_2},\\
            (a^ib^jm_2-m_1)t_{i,j}&=0 & a^ib^j\neq \frac{m_1}{m_2},
        \end{align*}
        so $s_0f(x,y)=v_0f'(x,y)$. Since $B \in \GL(2,\mc{O}_W)$ and $u=0$, $s_0,v_0$ must both be non-zero. Thus if $E_A\simeq E_{A'}$, there is some constant $c \in \C^*$ so that $f=cf'$.
        \end{proof}

        Now suppose that $E$ is given by an extension $$\begin{tikzcd}
            0\arrow[r] & \lambda_1\arrow[r] & E \arrow[r] & \lambda_2\arrow[r] & 0
        \end{tikzcd}$$
        with $H^1(X,\lambda_1\otimes \lambda_2^{-1})\neq 0$ and $\deg(\lambda_1)<\deg(\lambda_2)$.
        \begin{lemma}
            Let $E$ be a filtrable rank-2 vector bundle on a diagonal Hopf surface $X$ with $c_2(E)=0$. Then $E$ is not $\mu$-stable.
        \end{lemma}
        \begin{proof}
            Since $E$ is filtrable, it has a sub-line bundle. Let $\lambda\in \Pic(X)$ have maximal degree among line bundles with a non-zero map to $E$. Then $\lambda\hookrightarrow E$ will have torsion-free quotient given by a line bundle $L$, so $E$ is given by the extension $$\begin{tikzcd}
                0 \arrow[r] & \lambda \arrow[r] & E \arrow[r] & L \arrow[r] & 0.
            \end{tikzcd}$$
            If $\deg(\lambda)\geq \mu(E)$, then we are done, so assume that $\deg(\lambda)<E$. The extension cannot be split, as otherwise $L$ would admit a map to $E$, contradicting our assumption as $\deg(L)=2\mu(E)-\deg(\lambda)>\mu(E)$. Since $E$ is given by a non-split extension of $L$ by $\lambda$, $H^1(X,\lambda\otimes L^{-1})\neq0$. Using Table \ref{line-bundles-Hopf}, we see that since $\deg(\lambda\otimes L^{-1})<0$, $H^1(X,\lambda\otimes L^{-1})\simeq H^2(X,\lambda\otimes L^{-1})\simeq H^0(X, K_X\otimes L\otimes \lambda^{-1})$. If we tensor the exact sequence defining our extension by $L_b^{-1}\otimes \lambda^{-1}$ and take cohomology, we obtain the exact sequence $$\begin{tikzcd}0 \arrow[r] & H^0(X,L_b^{-1})\arrow[r] & H^0(X,E\otimes L_b^{-1}\otimes \lambda^{-1})\arrow[r] & H^0(X,L\otimes \lambda^{-1}\otimes L_b^{-1})\arrow[r] & H^1(X,L_b^{-1}).\end{tikzcd}$$
            Again, consulting Table \ref{line-bundles-Hopf} gives $H^1(X,L_b^{-1})=0$ and $h^0(X,L\otimes \lambda^{-1}\otimes L_b^{-1})\geq h^0(X,L\otimes \lambda^{-1}\otimes K_X)>0,$ so we must have $H^0(X,E\otimes \lambda^{-1}\otimes L_b^{-1})\neq 0$. Then the line bundle $\lambda\otimes L_b$ admits a non-zero map to $E$, contradicting the maximality of $\deg(\lambda)$ since $\deg(\lambda\otimes L_b)=\deg(\lambda)+1.$ Therefore $\deg(\lambda)\geq \mu(E)$ and $E$ is not stable.
        \end{proof}
        \begin{proposition}\label{decomposable-extension}
            Let $\lambda_1,\lambda_2 \in \Pic(X)$ be line bundles on a diagonal Hopf surface $X$ such that $\deg(\lambda_1)<\deg(\lambda_2)$. Then every extension $$\begin{tikzcd}
                0 \arrow[r] & \lambda_1 \arrow[r] & E \arrow[r] & \lambda_2 \arrow[r] & 0
            \end{tikzcd}$$ has $E$ isomorphic to a direct sum of line bundles.
        \end{proposition}
        \begin{proof}
            By the previous Lemma, $E$ is not stable, so either it is a direct sum of line bundles, or it is indecomposable of the type appearing in Proposition \ref{diag-indecomposable}. As shown in Proposition \ref{diag-indecomposable}, any such indecomposable bundle will have a unique line bundle $L$ with $\deg(L)\geq \mu(E)$ so that every map from a line bundle to $E$ factors through $L$. Since $E$ admits a map from $\lambda_1$ with torsion-free quotient and $\deg(\lambda_1)<\mu(E)$, the indecomposable option is not possible. Therefore $E$ is isomorphic to a direct sum of line bundles.
        \end{proof}
        \begin{remark}
             If $X$ is generic, the non-split extension $$\begin{tikzcd}
                0 \arrow[r] & \mc{O}_X \arrow[r] & E \arrow[r] & L_a^k\otimes L_b^\ell \arrow[r] & 0
            \end{tikzcd}$$ with $k\geq 1, \ell\geq 1$ will have $E\simeq L_a^k\oplus L_b^\ell$. For the other diagonal cases, determining the isomorphism class of an extension of the type in Proposition \ref{decomposable-extension} is more subtle, but we also have that if $X$ is classical, any non-split extension $$\begin{tikzcd}
                0 \arrow[r] & \mc{O}_X \arrow[r] & E \arrow[r] & L_b^k \arrow[r] & 0
            \end{tikzcd}$$ for $k\geq 2$ will be given as the pullback of an extension of line bundles on $\mb{P}^1$.
        \end{remark}

\subsection{Extensions of line bundles for a non-diagonal Hopf surface.}\label{exceptional-extension}
Let $X$ be the exceptional Hopf surface given by the map $\varphi:(x,y)\mapsto (b^rx+y^r,b y)$, and let $L$ be a line bundle with $h^1(X,L)=1$. We first consider the case where $\deg(L)\geq 0$.

\begin{proposition}\label{upper-extension}
    Let $E$ be the unique non-split extension $$\begin{tikzcd}
    0 \arrow[r] & (L_b)^{rk+\ell} \arrow[r] & E \arrow[r] & \mc{O}_X \arrow[r] & 0,
\end{tikzcd}$$
where $k$ and $\ell$ are integers with $k\geq 0$ and $0\leq \ell <r$. Then $$A=\begin{pmatrix}
    b^{rk+\ell} & x^ky^\ell\\ 0 & 1
\end{pmatrix}$$
is a factor of automorphy representing $E$.
\end{proposition}
\begin{remark}
    Note that while the resulting classification of indecomposable bundles is the same as \cite[Proposition 6.5]{Mall93}, the formula $$B=\begin{pmatrix}
        b^m & y^m\\ 0 & 1
    \end{pmatrix}$$ given for the factor of automorphy of a non-split extension $$\begin{tikzcd}
        0 \arrow[r] & (L_b)^m \arrow[r] & E \arrow[r] & \mc{O}_X \arrow[r] & 0
    \end{tikzcd}$$
    is incorrect when $m\geq r$, as in this case a splitting can be constructed for the bundle represented by $B$ using the nowhere-vanishing section $(xy^{m-r},b^{m-r})$.
\end{remark}
\begin{proof}
    Clearly the vector bundle $E_A$ represented by the factor of automorphy $A$ is an extension of the trivial line bundle by $(L_b)^{rk+\ell}$, so $E_A$ must be isomorphic to one of $E$ and $\mc{O}_X\oplus (L_b)^{rk+\ell}$. In order to prove that $E_A\simeq E$, we will show that $h^0(X,E_A)\neq h^0(X,\mc{O}_X\oplus (L_b)^{rk+\ell})=2$. 
    
    Holomorphic sections of $E_A$ are in one-to-one correspondence with pairs of holomorphic functions $f,g \in \mc{O}_W$ satisfying $$\begin{pmatrix}
        f(b^rx+y^r, b y)\\
        g(b^rx+y^r, b y)
    \end{pmatrix}=A\begin{pmatrix}
        f(x,y)\\ g(x,y)
    \end{pmatrix}=\begin{pmatrix}
        b^{rk+\ell}f(x,y)+x^ky^\ell g(x,y)\\
        g(x,y)
    \end{pmatrix}.$$
    Clearly $g(x,y)=c$ for some constant $c$ by the second entry. Let $f(x,y)=\sum\limits_{i,j}a_{i,j}x^iy^j$ be the Taylor series for $f$. Then the first entry of our equation can be rewritten as $$\sum\limits_{i,j}\sum\limits_{s=0}^i\binom{i}{s}a_{i,j}b^{rs+j}x^sy^{r(i-s)+j}=cx^ky^\ell+b^{rk+\ell}\sum\limits_{i,j}a_{i,j}x^iy^j$$ and re-indexing by $\alpha=s$, $\beta=r(i-s)+j$, $t=i-s$ gives $$\sum\limits_{\alpha,\beta}\sum\limits_{t=0}^{\lfloor \frac{\beta}{r}\rfloor}\binom{\alpha+t}{\alpha}a_{\alpha+t,\beta-rt}b^{r(\alpha-t)+\beta}x^\alpha y^\beta =cx^ky^\ell+b^{rk+\ell}\sum\limits_{\alpha,\beta}a_{\alpha,\beta}x^\alpha y^\beta$$
    (Note that the index range for $t$ immediately follows from the restriction that $\alpha\geq$ and $\beta \geq 0$). By comparing coefficients in the above equation, we get the following recurrence relation on the $a_{\alpha,\beta}$s:
    \begin{align*}
        b^{rk+\ell}a_{k,\ell}&=c+b^{rk+\ell}a_{k,\ell},\\
        \left(b^{rk+\ell}-b^{r\alpha+\beta}\right)a_{\alpha,\beta}&=\sum\limits_{t=1}^{\lfloor \frac{\beta}{r}\rfloor}\binom{\alpha+t}{\alpha}b^{r(\alpha-t)+\beta}a_{\alpha+t,\beta-rt}, && (\alpha,\beta)\neq (k,\ell).
    \end{align*}
    The $x^ky^\ell$ coefficient immediately gives that $c=0$, so the pair $(f,0)$ gives a section of $E_A$ if and only if $f$ gives a section of the line bundle $(L_b)^{rk+\ell}$. Thus $h^0(X,E_A)=h^0(X,(L_b)^{rk+\ell})=1\neq h^0(X,\mc{O}_X\oplus (L_b)^{rk+\ell}),$ so $E_A$ is isomorphic to $E$.
\end{proof}

\begin{proposition}
    Suppose that $E$ is the unique non-split extension $$\begin{tikzcd}
        0 \arrow[r] & \mc{O}_X \arrow[r] & E \arrow[r] & (L_b)^{rk+\ell} \arrow[r] & 0,
    \end{tikzcd}$$ where $k$ and $\ell$ are integers satisfying $0\leq \ell<r<rk+\ell$. Then $E$ is isomorphic to $F$, where $F$ is the unique non-split extension $$\begin{tikzcd}
    0 \arrow[r] & (L_b)^{rk} \arrow[r] & F \arrow[r] & (L_b)^\ell \arrow[r] & 0.
\end{tikzcd}$$

Unlike in the diagonal case, extensions of $L$ by $\mc{O}_X$ with $\deg(L)<0$ are also indecomposable, though they are still not stable.

\end{proposition} 
\begin{proof}
    Since $h^1(X,(L_b)^{-rk-\ell})=1$ and $F$ is indecomposable, it is enough to show that $\det(F)=\det(E)$ and that $F$ has a nowhere vanishing section. Using Proposition \ref{upper-extension}, it is easy to show that $F$ is represented by the factor of automorphy $$A=\begin{pmatrix}
        b^{rk} & x^{k-1}y^{r-\ell}\\ 0 & b^\ell
    \end{pmatrix}.$$
    Recall that a pair of functions $(f(x,y),g(x,y))$ define a section of $F$ if and only if $$\begin{pmatrix}
        f(b^rx+y^r,by)\\ g(b^rx+y^r,by)
    \end{pmatrix}=A\begin{pmatrix}
        f(x,y)\\g(x,y)
    \end{pmatrix}=\begin{pmatrix}
        b^{rk}f(x,y)+x^{k-1}y^{r-\ell}g(x,y)\\ b^\ell g(x,y)
    \end{pmatrix}.$$ Looking first at the second component, $g$ must represent a section of $(L_b)^\ell$, so by \cite[Proposition 3.6]{Mall91} it will have the form $g(x,y)=cy^\ell$ for some constant $c\in \C$. Writing $f(x,y)=\sum\limits_{i,j}a_{i,j}x^iy^j$ and re-indexing in a similar fashion to Proposition \ref{upper-extension} gives the first component of our section equation as $$\sum\limits_{\alpha,\beta}\sum\limits_{t=0}^{\lfloor \frac{\beta}{r}\rfloor}\binom{\alpha+t}{\alpha}a_{\alpha+t,\beta-rt}b^{r(\alpha-t)+\beta}x^\alpha y^\beta=cx^{k-1}y^r+\sum\limits_{\alpha,\beta}a_{\alpha,\beta}b^{rk}x^\alpha y^\beta,$$
    leading to the recurrence relation \begin{align*}
        \binom{k}{k-1}a_{k,0}b^{r(k-1)}&=c,\\
        (b^{rk}-b^{r\alpha+\beta})a_{\alpha,\beta}&=\sum\limits_{t=1}^{\lfloor\frac{\beta}{r}\rfloor}\binom{\alpha+t}{\alpha}b^{r(\alpha-t)+\beta}a_{\alpha+t,\beta-rt}, && (\alpha,\beta)\neq (k-1,r).
    \end{align*}
    Suppose $(\alpha,\beta)$ are such that $r\alpha+\beta\neq rk$. Checking the recurrence relation for $(\alpha',\beta')\in \{(u,r(\alpha-u)+\beta), u=0,1,\ldots, \lfloor\frac{\beta}{r}\rfloor\}$, an induction argument gives $a_{\alpha,\beta}=0$, so we can assume $\beta=r(k-\alpha)$. Under this assumption, the recurrence relation simplifies to a linear system \begin{align*}
        c&=kb^{r(k-1)}a_{k,0},\\
        0&=\sum\limits_{t=1}^{k-\alpha}\binom{\alpha+t}{\alpha}b^{r(k-t)}a_{\alpha+t, r(k-\alpha-t)}, && \alpha=0,\ldots, k-2,
    \end{align*}
    which we can write in matrix form as $$\begin{pmatrix}
        0 & \binom{1}{0}b^{r(k-1)} & \binom{2}{0}b^{r(k-2)}& \ldots & \binom{k}{0}\\
        0 & 0 & \binom{2}{1}b^{r(k-1)}& \ldots & \binom{k}{1}b^r\\
        \vdots & \vdots & \vdots & \ddots & \vdots\\
        0 & 0 & 0 & \ldots & \binom{k}{k-1}b^{r(k-1)}
    \end{pmatrix}\begin{pmatrix}
        a_{0,rk}\\
        a_{1,r(k-1)}\\
        \vdots\\
        a_{k,0}
    \end{pmatrix}=\begin{pmatrix}
        0\\0\\ \vdots \\ c
    \end{pmatrix}.$$
    Based on the form of this matrix, there is a pair $$(f_0(x,y),g_0(x,y))=\left(x^k+y^r\sum\limits_{\alpha=0}^{k-1} a_{\alpha,r(k-\alpha)}x^\alpha y^{r(k-\alpha-1)}, kb^{r(k-1)}y^\ell\right)$$ representing a section of $F$. Note that $V(g_0)=\{(x,0)\in W\}$ and $f_0|_{V(g_0)}=x^k$, so $V(f_0,g_0)=\emptyset$.
\end{proof}

If $X$ is an exceptional Hopf surface, its tangent bundle is isomorphic to the unique non-split extension $$\begin{tikzcd}
    0 \arrow[r] & (L_b)^r \arrow[r] & \mc{T}_X \arrow[r] & L_b \arrow[r] & 0.
\end{tikzcd}$$

\section{Co-Higgs bundles with $c_2=0$}\label{flat-co-Higgs}
We begin by showing that for any rank-2 trace-free co-Higgs bundle $(E,\Phi)$, the underlying vector bundle $E$ is filtrable.
\begin{proposition}
    Let $X$ be a primary Hopf surface, and let $\lambda \in \Pic(X)$ be such that $\Hom(\lambda,\mc{T}_X)\neq 0$. If $\eta:Y\to X$ is a double cover with $Y$ integral and embedded in $\mathrm{Tot}(\lambda)$, then $Y$ is birational to a Hopf surface.
\end{proposition}
\begin{proof}
    We can assume without loss of generality that $Y$ is normal, as the normalization of $Y$ will also be an integral double cover of $X$ embedded in $\mathrm{Tot}(\lambda')$ for some $\lambda'\in \Pic(X)$ with $\Hom(\lambda',\mc{T}_X)\neq 0$.
    
    For a Hopf surface $X$, the fact that $\Hom(\lambda, \mc{T}_X)\neq 0$ implies that $\deg(\lambda)<\deg(K_X^{-1})$, so $H^0(X,K_X^k\otimes \lambda^\ell)=0$ whenever $k\geq \ell$. Since $K_Y\simeq \eta^*(K_X\otimes \lambda)$, the plurigenera of $Y$ are given by $$P_k(Y)=h^0(X,(K_X\otimes \lambda)^k)+h^0(X,K_X^k\otimes \lambda^{k-1})=0.$$ Combining this with the fact that $b_2(X)=0$ allows us to calculate the invariants \begin{align*}
        h^{00}(Y)=1, && h^{01}(Y)=1, && h^{02}(Y)=0, && c_1^2(Y)=0, && c_2(Y)=0
    \end{align*}
    \cite[Section V.22]{Barth-Hulek-Peters-VandeVen03}. We also know that $Y$ is not projective as the algebraic dimension of $Y$ is equal to that of $X$. Note that every non-projective surface has non-negative second Chern number, so $Y$ is necessarily minimal. Therefore $Y$ is a minimal surface of class VII, and $h^{11}(Y)=c_2(Y)=0$. Finally, the ramification divisor of $\eta$ is a curve on $Y$, so $Y$ is a Hopf surface \cite{Bogomolov76}.
\end{proof}

\begin{proposition}\label{non-filtrable}
    Let $\eta:Y\to X$ be a double covering with $X$ a primary Hopf surface and $Y$ a Hopf surface. For any line bundle $\lambda \in \Pic(Y)$, $\eta_*\lambda$ is filtrable.
\end{proposition}
\begin{remark}
    An analogous result to this proposition appears in \cite[Proposition 3.14]{Moraru04} for the case of generic-type Hopf surfaces.
\end{remark}
\begin{proof}
    Any Hopf surface $Y$ has torsion Néron-Severi group, so $\chern(\eta_*\lambda)=2$ by Grothendieck-Riemann-Roch. If $\eta_*\lambda$ were stable (as is always the case for non-filtrable bundles of rank 2), then it would induce an irreducible $U(2)$ representation of $\pi_1(X)\simeq \Z$ via a Riemann-Hilbert correspondence. However, there is no such representation as $\Z$ is abelian, so $\eta_*\lambda$ is filtrable.
\end{proof}

Using Lemma \ref{factor-Higgs} and Proposition \ref{non-filtrable}, an important step in classifying co-Higgs bundles on a Hopf surface is classifying nowhere-vanishing Higgs fields $\phi\in H^0(X,\sEnd_0(E)\otimes \lambda)$, where $E$ is an extension of line bundles and $\lambda \in \Div(X)$ is such that $H^0(X,\mc{T}_X\otimes \lambda^{-1})\neq 0$. In the following, we will assume that $E$ is normalized so that $H^0(X,E)\neq 0$ but $H^0(X,E\otimes H)=0$ for any $H\in \Pic(X)$ with $\deg(H)<0$.

Recall that any extension of line bundles on a Hopf surface falls into one of the following three categories:
\begin{enumerate}
    \item $E\simeq \mc{O}_X\oplus H$ with $H \in \Div$ and $\deg(H)\leq 0$. In this case $\sEnd_0(E)\simeq \mc{O}_X\oplus H\oplus H^{-1}$;
    \item $E\simeq \mc{O}_X\oplus H$ with $H \in \Pic(X)\setminus \Div$ and $\deg(H)\leq 0$;
    \item $E$ is the non-split extension of $H$ by $\mc{O}_X$ determined by a non-zero class $s \in H^1(X,H^{-1})$ for $H\in \Div$ with $\deg(H)\leq 0$.
\end{enumerate}
The $\lambda$-twisted Higgs fields which do not vanish on any divisor can be described uniformly when the underlying bundle is of type (2), and when the underlying bundle is of type (1) for $\lambda \in \Div$ with $\deg(\lambda)\leq 0$. Note that in the following we will present $\phi$ so that it is determined by its parameters up to conjugation with an automorphism of $E$ unless otherwise indicated.

For bundles $E$ of type (2), $E$ admits a $\lambda$-twisted Higgs field which does not vanish on any divisor if and only if $\lambda\simeq \mc{O}_X$. In this case $$\phi=\begin{pmatrix}
    a & 0\\0 & -a
\end{pmatrix}$$ for some $a \in \C^*,$ and the direct summands are both $\phi$-invariant sub-line bundles.

If $E$ is of type (1) and $X$ is diagonal, in order for $\mc{O}_X\oplus H$ to admit a $\lambda$-twisted Higgs field that does not vanish on any divisor, a necessary condition is that either one of $\lambda, \lambda\otimes H, \lambda\otimes H^{-1}$ must be isomorphic to $\mc{O}_X$ or there are non-negative integers $k,\ell$ and some $\tau \in \C^*$ so that $$\lambda\oplus \lambda\otimes H\oplus \lambda\otimes H^{-1}\simeq (L_a)^k\oplus (L_b)^\ell\oplus L_\tau.$$
For an exceptional Hopf surface, the pair $(\lambda,H)$ must satisfy the first case.

If  $E$ is of type (1) and $\deg(\lambda)\leq 0$, then either $\phi$ is the unique nilpotent Higgs field $$\begin{pmatrix}
    0 & 1\\0 & 0
\end{pmatrix}$$ on $\mc{O}_X\oplus \lambda$ with $\mc{O}_X$ as an invariant sub-line bundle or $\lambda\simeq \mc{O}_X$ and $$\phi=\begin{pmatrix}
    a & 0\\0 & -a
\end{pmatrix}$$ for some $a\in \C^*$ with the two direct summands as $\phi$-invariant sub-line bundles.

For bundles of type (3), the behaviour differs depending on whether the Hopf surface is diagonal. For diagonal Hopf surfaces, we have the following result:
\begin{lemma}
    Let $X$ be a diagonal Hopf surface, and let $E$ be a non-split extension of $H$ by $\mc{O}_X$, where $H \in \Div$ has $\deg(H)\leq 0$. If $\phi\in H^0(X,\sEnd_0(E)\otimes \lambda)$ is a $\lambda$-twisted Higgs field for some $\lambda \in \Div$, then $\phi=s\phi_0$ for some $s \in H^0(X,\lambda\otimes H^{-1})$, where $\phi_0:=(\iota\otimes \id_{H})\circ \rho$ with $$\begin{tikzcd}
    0 \arrow[r] & \mc{O}_X \arrow{r}{\iota} & E \arrow{r}{\rho} & H \arrow[r] & 0
\end{tikzcd}$$ being the short exact sequence corresponding to the extension class of $E$. The inclusion $\iota$ gives the unique $\phi_0$-invariant sub-line bundle.
\end{lemma}
\begin{proof}

    Consider the defining exact sequence for $E$ and tensor with $\lambda\otimes E^\vee$ for some $\lambda \in \Div$ to get $$\begin{tikzcd}
        0 \arrow[r] & \lambda\otimes E^\vee \arrow[r] & \sEnd(E)\otimes \lambda \arrow[r] & H\otimes \lambda \otimes E^\vee \arrow[r] & 0.
    \end{tikzcd}$$
    As $E$ is a rank-2 vector bundle, $E^\vee\simeq E\otimes \det(E)^{-1}\simeq E\otimes H^{-1}$. Applying this identification and taking the left exact sequence in cohomology gives $$\begin{tikzcd}
        0 \arrow[r] & H^0(X,\lambda\otimes H^{-1}\otimes E) \arrow{r}{H^0(\iota)} & H^0(\sEnd(E)\otimes,\lambda) \arrow{r}{H^0\rho} & H^0(X,\lambda\otimes E).
    \end{tikzcd}$$
    Note that $H^0(X,\lambda)$ injects into the image of $H^0(\rho)$ since $(\rho\otimes \id_\lambda)\circ (s\otimes \id_E)=\rho\otimes s$ for any $s \in H^0(X,\lambda)$. By Proposition \ref{diag-indecomposable}, $H^0(X,\lambda \otimes E)\simeq H^0(X,\lambda)$, so $H^0(\rho)$ is surjective and $$H^0(X,\sEnd_0(E)\otimes \lambda)\simeq H^0(X,\lambda\otimes H^{-1}\otimes E)\simeq H^0(X,\lambda\otimes H^{-1}).$$ 
    We can construct this isomorphism directly by mapping $s \in H^0(X,\lambda\otimes H^{-1})$ to $s \phi_0$.
\end{proof}

From this proposition it is clear that the Higgs field $s\phi_0$ will be non-vanishing away from a codimension-2 subset if and only if $s$ is a non-zero constant, so we must have $\lambda\simeq H$. For non-diagonal Hopf surfaces, the bundles of type (3) admit a wider variety of Higgs fields, so we will leave discussion of this case to section \ref{exceptional-Higgs}.

For $E$ of type (1) and $\deg(\lambda)>0$, we will look at each class of Hopf surfaces separately.

\subsection{Classical case}
    When $X$ is classical, $H^0(X,\mc{T}_X\otimes \lambda^{-1})\neq 0$ and $\deg(\lambda)>0$ imply $\lambda\simeq L_b$.  
    
    \begin{proposition}\label{Classical-extension}
        Let $X$ be a classical Hopf surface. The $L_b$-twisted Higgs bundles of the form $(\mc{O}_X\oplus H,\phi)$, up to conjugation by an automorphism of $\mc{O}_X\oplus H$, are as follows:
    \begin{enumerate}
        \item $H \simeq\mc{O}_X$ and  $$\phi=\begin{pmatrix}
        tx & sx+uy\\ y & -tx
    \end{pmatrix},$$ where $s,t,u \in \C$ and $s^2-4t^2u\neq0$, up to changing $t$ to $-t$. This Higgs field always has irreducible spectral curve.
    \item $H\simeq  L_b^{-1}$ and $$\phi=\begin{pmatrix}
        0 & \beta\\ 1 & 0
    \end{pmatrix}$$ with $\beta \in H^0(X,(L_b)^2)$. This Higgs field will have irreducible spectral curve when $\beta$ is not a square. When $\beta=t^2$ for some $t \in H^0(X,L_b)$, the $\phi$-invariant sub-line bundles are $L_b^{-1}$ under the inclusions $(\pm t, 1).$
    \item $H\simeq  (L_b)^{-k}$ for some $k\geq 0$ and $$\phi=\begin{pmatrix}
        ux+vy & (\overline{v}x-\overline{u}y)^{1+k}\\ 0 & -ux-vy
    \end{pmatrix}$$ for $u,v \in \C$ which are not both zero. The $\phi$-invariant sub-line bundles are the first direct summand $\mc{O}_X$ and $(L_b)^{-1-k}$ under the inclusion $((\overline{v}x-\overline{u}y)^{1+k}, -2ux-2vy).$
    \end{enumerate} 
    \end{proposition}
    \begin{remark}
        Computations for cases (1) and (2) of this proposition appear in \cite[Examples 3-4]{Boulter-Moraru25}
    \end{remark}
    \begin{proof}
        First consider the case where $H\simeq L_b^{-k}$ with $k>1$. Then $\sEnd_0(E)\otimes L_b\simeq L_b^{1-k}\oplus L_b\oplus L_b^{1+k}$. The first component $L_b^{1-k}$ has no sections, so $\phi$ is of the form $$\phi=\begin{pmatrix}
            f & g\\ 0 & -f
        \end{pmatrix}.$$
        Automorphisms of $E$ are of the form $$\begin{pmatrix}s & \beta\\ 0 & t\end{pmatrix},$$
        where $s,t \in \C^*$ and $\beta \in H^0(X,L_b^{k})$. Conjugating $\phi$ by such an automorphism gives $$\begin{pmatrix}
            f & \frac{s}{t}g-\frac{2}{t}\beta f\\0 & -f
        \end{pmatrix}.$$
        Note that $g$ and $f\cdot H^0(X,L_b)$ together span $H^0(X,L_b^2)$ by the assumption that $V(f,g)=\emptyset$, so $s,\beta,t$ can be chosen to make the upper-right component any polynomial whose zero set does not contain the zero set of $f$. Thus, if $f=ux+vy$, we can choose a unique representative $\overline{v}x-\overline{u}y$ for $g$ to get $\phi$ of type (3).

        If $H\simeq L_b^{-1}$, $\sEnd_0(E)\otimes L_b\simeq \mc{O}_X\oplus L_b\oplus L_b^2$. If the $\mc{O}_X$-component is zero, the same argument as in the previous case shows that $\phi$ is of type (3) with $k=1$. Now assume $$\phi=\begin{pmatrix}
            \alpha & \beta\\ c & -\alpha
        \end{pmatrix}$$
        with $c \in \C^*$, $\alpha\in H^0(X,L_b)$, $\beta \in H^0(X,L_b^2$. Conjugating with an automorphism $$\begin{pmatrix}
            c & -\alpha\\0 & 1
        \end{pmatrix}$$
        gives $\phi$ of type (2).

        Finally, if $k=0$, $\sEnd_0(E)\otimes L_b$ is just $\mathfrak{sl}_2\otimes L_b$. If $\phi$ can be made upper-triangular, a similar argument to when $k>1$ shows that $\phi$ is of type (3).

        If the $x$- and $y$-components cannot be made simultaneously upper-triangular, then $\phi$ is conjugate to $\begin{pmatrix}
            f & g\\cy & -f
        \end{pmatrix}$ for some $f,g \in H^0(X,L_b)$ and $c \in \C^*$. By conjugating with $$\begin{pmatrix}
            c & -f(1,0)\\0 & 1
        \end{pmatrix},$$
        $\phi$ transforms to $$\begin{pmatrix}
            tx & sx+uy\\y & -tx
        \end{pmatrix}$$
        for $s,t,u \in \C$ where $s$ and $t$ are not both zero by the fact that $\phi$ is nowhere-vanishing. If $s^2-4t^2u=0$, then $t\neq 0$ and $u=\frac{s^2}{4t^2}$. In this case, one can show that the $x$- and $y$-components of $\phi$ can be made simultaneously upper triangular, so if this is not the case $\phi$ is of type (1). Finally, note that conjugating $\phi$ of type (1) with $\begin{pmatrix}
            s & 2tu\\2t & s
        \end{pmatrix}$ has the effect of multiplying $t$ by $-1$.

        In all cases with reducible spectral cover (i.e. cases with $\det(\phi)$ a perfect square), the invariant sub-line bundles can be computed via standard eigenvector computations.
    \end{proof}

    \subsection{The Resonant case}
    When $X$ is resonant $H^0(X,\mc{T}_X\otimes \lambda^{-1})\neq 0$ and $\deg(\lambda)>0$ imply $\lambda\simeq (L_b)^\ell$ for some $\ell\in \Z$ with $0<\ell\leq r$. 

    Recall that for $s \in H^0(X,(L_b)^k)$ for some $k\geq 0$, the only way for $s$ to be non-zero on points of the form $(x,0)$ is if $k$ is a multiple of $r.$ 

    \begin{proposition}
        Let $X$ be a resonant Hopf surface. The $L_b^{r}$-twisted Higgs bundles of the form $(\mc{O}_X\oplus H,\phi)$, up to conjugation by an automorphism of $\mc{O}_X\oplus H$, are as follows:
    \begin{enumerate}
        \item $H\simeq \mc{O}_X$ and $$\phi=\begin{pmatrix}
        tx & sx+uy^r\\ y^r & -tx
    \end{pmatrix},$$ where $s^2-4t^2u\neq 0$, up to replacing $t$ with $-t$. This Higgs field always has irreducible spectral curve;
    \item $H\simeq (L_b)^{-r}$ and $$\phi=\begin{pmatrix}
        0 & \beta\\1 & 0
    \end{pmatrix}$$ with $\beta\in H^0(X,(L_b)^{2r})$. This Higgs field will have irreducible spectral curve when $\beta$ is not a square. When $\beta=t^2$ for some $t \in H^0(X, (L_b)^r)$, the invariant sub-line bundles are $(L_b)^{-r}$ under the inclusions $(\pm t, 1)$.
    \item $H\simeq (L_b)^{-kr}$ for some $k\geq0$ and $$\phi=\begin{pmatrix}
        ux+vy^r & (\overline{v}x-\overline{u}y^r)^{1+k}\\ 0 & -ux-vy^r
    \end{pmatrix}$$ for $u,v \in \C$ not both zero. The $\phi$-invariant sub-line bundles are the first direct summand $\mc{O}_X$ and $(L_b)^{-r(1+k)}$ under the inclusion $((\overline{v}x-\overline{u}y^r)^{1+k}, -2ux-2vy^r).$
    \item $H\simeq (L_b)^{-k}$ with $0<k<r$ and $$\phi=\begin{pmatrix}
        s x & \beta\\ y^{r-k} & -s x
    \end{pmatrix}$$ for  $s \in \C^*$ and $\beta\in H^0(X,(L_b)^{r+k})$. This Higgs field will have irreducible spectral curve if there is no constant $t \in \C$ so that $\beta=2stxy^k+t^2y^{r+k}$. For $\beta=2stxy^k+t^2y^{r+k}$, the $\phi$-invariant sub-line bundles are $(L_b)^{-k}$ under the inclusion $(ty^k, -1)$ and $(L_b)^{-r}$ under the inclusion $(2s x+ty^r, y^{r-k}).$
    \item $H\simeq (L_b)^{-k}$ for $k>0$ with $k$ not a multiple of $r$ and $$\phi=\begin{pmatrix}
        \alpha & y^{r+k}\\0 & -\alpha
    \end{pmatrix}$$ for $\alpha \in H^0(X,(L_b)^r)$ with $\alpha(1,0)\neq 0$. The $\phi$-invariant sub-line bundles are the first direct summand $\mc{O}_X$ and $(L_b)^{-r-k}$ under the inclusion  $(y^{r+k}, -2\alpha)$.
    \end{enumerate}
    \end{proposition}
    \begin{proof}
        The cases where $H\simeq (L_b)^{-kr}$ for $k\geq 0$ are analogous to Proposition \ref{Classical-extension} with $y$ replaced by $y^r$. If $H\simeq L_b^{-k}$ for some $k>0$ which is not a multiple of $r$, then $\sEnd_0(E)\otimes L_b^r\simeq L_b^{r-k}\oplus L_b^r\oplus L_b^{r+k}$. In particular, since any section of $L_b^{r+k}$ or $L_b^{r-k}$ will vanish at $(1,0)$, the fact that $\phi$ is nowhere vanishing implies that the $L_b^{r}$ component does not vanish at $(1,0)$. If the $L_b^{r-k}$ component of $\phi$ is non-zero (in which case $k<r$), then $\phi$ is of the form $$\begin{pmatrix}
            \alpha & \beta\\cy^{r-k} & -\alpha
        \end{pmatrix}$$ for some $c \in \C^*$, $\alpha \in H^0(X,L_b^r)$, $\beta\in H^0(X,L_b^{r+k})$. Conjugating by $$\begin{pmatrix}
            c & -\alpha(0,1)y^k\\0 & 1
        \end{pmatrix}$$
        gives $\phi$ of type (4).

        If instead the $L_b^{r-k}$-component of $\phi$ is zero, we will have $\phi$ of the form $$\begin{pmatrix}
            \alpha & \beta\\0 & -\alpha
        \end{pmatrix}$$ for $\alpha \in H^0(X,L_b^r)$ with $\alpha(1,0)\neq 0$ and $\beta \in H^0(X,L_b^{r+k})$. Conjugating with an automorphism of the form $$\begin{pmatrix}
            s & f\\ 0 & 1
        \end{pmatrix}$$ with $s\in \C^*$ and $f \in H^0(X,L_b^k)$ gives $$\begin{pmatrix}
            \alpha & s\beta-2f\alpha\\
             0 & -\alpha
        \end{pmatrix}.$$
        Notice that the subspaces $\beta\cdot \C$ and $\alpha\cdot H^0(X,L_b^k)$ together span $H^0(X,L_b^{r+k})$, so appropriate choices of $s$ and $f$ can make the top-right component into any section that does not have $\alpha$ as a factor. Since $\alpha(1,0)\neq 0$, we can take the top-right component to be $y^{r+k}$, making $\phi$ of type (5).
    \end{proof}

    \begin{proposition}
        Let $X$ be a resonant Hopf surface, and let $\ell$ be an integer with $0<\ell<r$. The $L_b^{\ell}$-twisted Higgs bundles on $X$ of the form $(\mc{O}_X\oplus H, \phi)$, up to conjugation by an automorphism of $\mc{O}_X\oplus H$ are as follows:
        \begin{enumerate}
            \item $H\simeq (L_b)^{\ell-kr}$ for some $k>0$ and $$\phi=\begin{pmatrix}
        s y^\ell & x^k\\ 0 & -s y^\ell
        \end{pmatrix}$$
        for some $s \in \C^*$. The $\phi$-invariant sub-line bundles are the first direct summand $\mc{O}_X$ and $(L_b)^{-kr}$ under the inclusion $(x^k,-2s y^\ell)$.
        \item $H\simeq (L_b)^{-\ell}$ and $$\phi=\begin{pmatrix}
        0 & \beta\\ 1 & 0
        \end{pmatrix}$$ for $\beta \in H^0(X,(L_b)^{2\ell})$. This Higgs field will have irreducible spectral curve if $\beta$ is not a square (only possible when $2\ell=r$). When $\beta=t^2$ for some $t \in H^0(X,(L_b)^\ell)$, the $\phi$-invariant sub-line bundles are $L_b^{-\ell}$ under the inclusions $(\pm t, 1).$
        \item (only if $\frac{r}{2}<\ell<r$) $H\simeq (L_b)^{\ell-r}$ and $$\phi=\begin{pmatrix}
        0 & \beta\\ y^{2\ell-r} & 0
    \end{pmatrix}$$ for $\beta\in H^0(X,(L_b)^r)$ with $\beta(1,0)\neq 0$. This Higgs field will always have irreducible spectral curve.
        \end{enumerate}
    \end{proposition}
    \begin{proof}
        Suppose that $\mc{O}_X\oplus L_b^{-n}$ admits a nowhere-vanishing $L_b^\ell$-twisted Higgs field for some $n \in \Z_{\geq 0}$. Since we have assumed $0<\ell<r$, this implies that either $n=\ell$ so that $\sEnd_0(E)\otimes L_b^{\ell}\simeq \mc{O}_X\oplus L_b^{\ell}\oplus L_b^{2\ell}$ or $n=kr-\ell$ for some $k\geq 0$, so that $\sEnd_0(E)\otimes L_b^\ell\simeq L_b^{2\ell-kr}\oplus L_b^\ell\oplus L_b^{kr}.$

        In the case $n=\ell$, $\phi$ will be of the form $$\begin{pmatrix}
            \alpha & \beta\\c & -\alpha
        \end{pmatrix}$$ for some $c\in \C^*$, $\alpha \in H^0(X,L_b^\ell)$, and $\beta \in H^0(X,L_b^{2\ell})$. Conjugating with the automorphism $$\begin{pmatrix}
            c & -\alpha\\ 0 & 1
        \end{pmatrix}$$ gives $\phi$ of type (2).

        In the case $n=kr-\ell$ for some $k>0$, there are two possibilities depending on whether the $L_b^{2\ell-kr}$-component of $\phi$ is non-zero. If the $L_b^{2\ell-kr}$-component is zero, then $\phi$ is of the form $$\begin{pmatrix}
            sy^\ell & \beta\\0 & -sy^\ell
        \end{pmatrix}$$ for $s \in \C^*$ and $\beta \in H^0(X,L_b^{kr})$ with $\beta(1,0)\neq0$. Conjugating with an automorphism of the form $$\begin{pmatrix}
            t & fy^{r-\ell}\\ 0 & 1
        \end{pmatrix}$$ for $t \in \C^*$ and $f \in H^0(X,L_b^{(k-1)r})$ gives $$\begin{pmatrix}
            sy^\ell & t\beta-2sfy^r\\
             0 & -sy^\ell
        \end{pmatrix}.$$
        Note that the subspaces $\beta\cdot \C$ and $y^r\cdot H^0(X,L_b^{(k-1)r})$ span $H^0(X,L_b^{kr})$, so the top-right component of $\phi$ can be conjugated to any section which does not vanish on $(0,1)$. Choosing the top-right component to be $x^k$ gives $\phi$ of type (1).

        If the $L_b^{2\ell-kr}$-component is non-zero (and $n\neq -\ell)$, then we must have $2\ell-kr>0$, so that $k=1$ and $\ell>\frac{r}{2}$. In this case, $\phi$ is of the form $$\begin{pmatrix}
            \alpha & \beta\\  cy^{2\ell-r} & -\alpha
        \end{pmatrix}$$ for $c \in \C^*$, $\alpha \in H^0(X,L_b^\ell)$ and $\beta \in H^0(X,L_b^r)$ with $\beta(0,1)\neq 0$. Conjugating with the automorphism $$\begin{pmatrix}
            c & -\alpha\\ 0 & 1
        \end{pmatrix}$$ gives $\phi$ of type (3).
    \end{proof}

    \subsection{The Hyper-resonant case}
    When $X$ is a hyper-resonant Hopf surface, $H^0(X,\mc{T}_X\otimes \lambda^{-1})\neq 0$ and $\deg(\lambda)>0$ imply $\lambda\simeq L_a$, $\lambda\simeq L_b$, or $\lambda\simeq L_a\otimes (L_b)^{-\ell}$ for some integer $\ell$ with $0<\ell<\frac{r_2}{r_1}$.

    \begin{proposition}\label{hyper-extension}
        Let $X$ be a hyper-resonant Hopf surface given by the contraction $(x,y)\mapsto (ax,by)$ where $a^{r_1}=b^{r_2}$ and $r_1\leq r_2$. 
        \begin{enumerate}
            \item The nowhere-vanishing $L_a$-twisted Higgs bundles of the form $(\mc{O}_X\oplus H,\phi)$, up to conjugation by an automorphism of $\mc{O}_X\oplus H$, are as follows:
        \begin{enumerate}
            \item $H\simeq L_a^{-1}$ and $$\phi=\begin{pmatrix}
                0 & \beta\\ 1 & 0
            \end{pmatrix}$$ form some $\beta \in H^0(X,L_a^2)$. This Higgs field has irreducible spectral curve when $r_1=2$ and $\beta$ is not a square. When $\beta=t^2$ for some $t \in H^0(X,L_a)$, the $\phi$-invariant sub-line bundles are $(L_a)^{-1}$ under the inclusions $(\pm t, 1)$.
            \item $H\simeq L_a\otimes L_b^{-k}$ with $k\geq \frac{r_2}{r_1}$ and $$\phi=\begin{pmatrix}
                sx & y^k\\0 & -sx
            \end{pmatrix}$$ for some $s \in \C^*$. The $\phi$-invariant sub-line bundles are the first direct summand $\mc{O}_X$ and $(L_b)^{-k}$ under the inclusion $(y^k, -2sx).$
            \item (Only for $r_1=2$) $H\simeq L_a\otimes L_b^{-k}$ with $\frac{r_2}{2}\leq k<r_2$ and $$\phi=\begin{pmatrix}
                sx & ty^k\\ y^{r_2-k} & -sx
            \end{pmatrix}$$
            for some $s\in \C^*$ and $t\in \C$. This Higgs field has irreducible spectral curve if $t\neq 0$. For $t=0$, the $\phi$-invariant sub-line bundles are the second direct summand $H$ and $(L_a)^{-1}$ under the inclusion $(2sx, y^{r_2-k})$.
            \item (Only for $r_1>2$) $H\simeq L_b^k\otimes L_a^{-1}$ with $0<k< \frac{r_2}{r_1}$ and $$\phi=\begin{pmatrix}
                sx & 0\\y^k & -sx
            \end{pmatrix}$$ for some $s\in \C^*$. The $\phi$-invariant sub-line bundles are the second direct summand $H$ and $(L_a)^{-1}$ under the inclusion $(2sx,y^k)$.
        \end{enumerate}
        \item The nowhere-vanishing $L_b$-twisted Higgs bundles of the form $(\mc{O}_X\oplus H, \phi)$, up to conjugation by an automorphism of $\mc{O}_X\oplus H$, are as follows: 
        \begin{enumerate}
            \item $H\simeq L_b^{-1}$ and $$\phi=\begin{pmatrix}
                0 & \beta\\1 & 0
            \end{pmatrix}$$ for some $\beta \in H^0(X,L_b^2)$. This Higgs field has irreducible spectral curve when $r_2=2$ and $\beta$ is not a square. When $\beta=t^2$ for some $t \in H^0(X,L_b)$, the $\phi$-invariant sub-line bundles are $L_b^{-1}$ under the inclusions $(\pm t, 1).$
            \item $H\simeq L_b\otimes L_a^{-k}$ with $k>0$ and $$\phi=\begin{pmatrix}
                sy & x^k\\
                0 & -sy
            \end{pmatrix}$$ for some $s \in \C^*$. The $\phi$-invariant sub-line bundles are the first direct summand $\mc{O}_X$ and $(L_a)^{-k}$ under the inclusion $(x^k, -2sy).$
            \item (Only for $r_1=r_2=2$) $H\simeq L_b\otimes L_a^{-1}$ and $$\phi=\begin{pmatrix}
                sy & tx\\ x & -sy
            \end{pmatrix}$$ for some $a\in \C^*$ and $b \in \C$. This Higgs field is irreducible if $t\neq 0$. When $t=0$, the $\phi$-invariant sub-line bundles are the second direct summand $H$ and $(L_b)^{-1}$ under the inclusion $(2sy, x)$.
        \end{enumerate}
        \item If $\ell$ is an integer with $0<\ell<\frac{r_2}{r_1}$, the $\lambda=L_a\otimes L_b^{-\ell}$-twisted Higgs bundles are all of the form $$\left(\mc{O}_X\oplus (L_b^\ell\otimes L_a^{-1}), \begin{pmatrix}
            0 & \beta\\1 & 0
        \end{pmatrix}\right)$$ for some $\beta \in H^0(X,\lambda^2)$. (If $r_1=2$, then $\beta=ty^{r_2-2\ell}$ for some $t \in \C$, and $\beta=0$ otherwise.) This Higgs field has irreducible spectral curve when $r_2$ is odd and $\beta\neq 0$. When $\beta=u^2$ for some $u \in H^0(X,L_{-1}\otimes \lambda)$, the phi-invariant sub-line bundles are $\lambda^{-1}$ under the inclusions $(\pm u, 1)$.
        \end{enumerate}
    \end{proposition}
    \begin{proof}
        Note that cases (1) and (2) are identical up to swapping $a$ with $b$ and $r_1$ with $r_2$, so we will only consider cases (1) and (3). (Note that for the $L_b$-twisted Higgs bundles there is no analogue of case (1d) since there is no integer $k$ satisfying $0<k<\frac{r_1}{r_2}\leq 1.$)

        Suppose that $\mc{O}_X\oplus H$ admits a nowhere-vanishing $\lambda$-twisted Higgs field. Then either $H\simeq \lambda^{-1}$ with the corresponding component non-zero, or one of $\lambda$, $\lambda\otimes H$, and $\lambda\otimes H^{-1}$ is of the form $L_a^k$ for some $k>0$ while another is of the form $L_b^m$ for some $m>0$.

        When $H\simeq \lambda^{-1}$ with the $\mc{O}_X$ term non-zero, $\phi$ has the form $$\begin{pmatrix}
            \alpha & \beta\\ c & -\alpha
        \end{pmatrix}$$ for some $c \in \C^*$, $\alpha \in H^0(X,\lambda)$, $\beta \in H^0(X,\lambda^2)$. conjugating with the automorphism $$\begin{pmatrix}
            c & -\alpha\\0 & 1
        \end{pmatrix}$$ gives $\phi$ of type (1a), (2a), or (3), depending on $\lambda$.

        Now suppose that $\lambda\simeq L_a$ and that either $L_a\otimes H$ or $L_a\otimes H^{-1}$ has the form $L_b^k$ for some $k>0$. If $L_a\otimes H^{-1}$ is of this form, then $H\simeq L_a\otimes L_b^{-k}$, and $\sEnd_0(\mc{O}_X\oplus H)\otimes L_a\simeq L_a\oplus (L_a^2\otimes L_b^{-k})\oplus L_b^k$. By our assumption that $\deg(H)\leq 0$, necessarily $k\geq \frac{r_2}{r_1}$. We will first assume that the $L_a^2\otimes L_b^{-k}$-component of $\phi$ is zero, so that $\phi$ is of the form $$\begin{pmatrix}
            sx & \beta\\0 & -sx
        \end{pmatrix}$$ for some $s \in \C^*$ and $\beta \in H^0(X,L_b^k)$ with $\beta(0,1)\neq 0$. If $k<r_2$, then $\beta=ty^k$ for some $t \in \C^*$, and conjugating with the automorphism $$\begin{pmatrix}
            1 & 0\\0 & t
        \end{pmatrix}$$ gives $\phi$ of type (1b). If $k\geq r_2$, conjugating with an automorphism of the form $$\begin{pmatrix}
            t & x^{r_1-1}f\\0 & 1
        \end{pmatrix}$$ for $t \in \C^*$ and $f \in H^0(X, L_b^{k-r_2})$ gives $$\begin{pmatrix}sx & t\beta-2sx^{r_1}f\\
        0 & -sx\end{pmatrix}.$$ The subspaces $\beta\cdot \C$ and $x^{r_1}\cdot H^0(X,L_b^{k-r_2})$ together span $H^0(X,L_b^k)$, so we can choose $t$ and $f$ so that $t\beta-2sx^{r_1}f=y^k$, so that $\phi$ is again of type (1b).

        If instead the $L_a^2\otimes L_b^{-k}$-component of $\phi$ does not vanish, then by Table \ref{line-bundles-Hopf} we must have $\lfloor \frac{2}{r_1}\rfloor+\lfloor\frac{-k}{r_2}\rfloor\geq 0$. Noting that $\lfloor \frac{-k}{r_2}\rfloor\leq -1$, this is only possible if $r_1=2$ and $\frac{r_2}{r_1}\leq k\leq r_2$. We can exclude the case with $r_1=2$ and $k=r_2$, as then $L_a\otimes H\simeq \mc{O}_X$, putting us in the case of type (1a). Now we have $$\phi=\begin{pmatrix}
            sx & ty^k\\uy^{r_2-k} & -sx
        \end{pmatrix}$$ for some $s,u \in \C^*$ and $t \in \C$. Conjugating with the automorphism $$\begin{pmatrix}
            u & 0\\0 & 1
        \end{pmatrix}$$ gives $\phi$ of type (1c).

        Now suppose that the $L_a$-twisted Higgs bundle $(\mc{O}_X\oplus H, \phi)$ is not of one of the above types, so that $H\simeq L_b^k\otimes L_a^{-1}$ for some $k>0$ and $H^0(X,L_a\otimes H^{-1})=0$. Since $\deg(H)\leq 0$, we must have $0<k<\frac{r_2}{r_1}$. We can disregard the case where $r_1=2$, as then $L_b^k\otimes L_a^{-1}\simeq L_a\otimes L^{k-r_2}$ and $r_2-k\geq \frac{r_2}{2}$, so that $\phi$ is of type (1b) or (1c). We can also exclude the case where $k=\frac{r_2}{r_1}$, as in that case $\mc{O}_X\oplus (L_b^{r_2/r_1}\otimes L_a^{-1})$ can be renormalized to $\mc{O}_X\oplus (L_a\otimes L_b^{-r_2/r_1}).$ Now $\phi$ is of the form $$\begin{pmatrix}
            sx & 0\\uy^k & -sx
        \end{pmatrix}$$ for some $s,u\in \C^*$, so conjugating by the automorphism $$\begin{pmatrix}
            u &0\\0 & 1
        \end{pmatrix}$$ gives $\phi$ of type (1d).

        Now all that remains is to check  that when $\ell$ is an integer with $0<\ell<\frac{r_2}{r_1}$, then any $L_a\otimes L_b^{-\ell}$-twisted Higgs bundle is of type (3). Since $H^0(X,L_a\otimes L_b^{-\ell})=0$, an $L_a\otimes L_b^{-\ell}$-twisted Higgs bundle not of type (3) will have $L_a\otimes L_b^{-\ell}\otimes H\simeq L_a^m$ and $L_a\otimes L_b^{-\ell}\otimes H^{-1}\simeq L_b^n$ for some $m,n >0$, or vice versa. If $H\simeq L_a^j\otimes L_b^k$ for some $j,k \in \Z$, then either $j\equiv 1 \mod r_1$ and $k\equiv \ell \mod r_2$, or $j\equiv r_1-1 \mod r_1$ and $k \equiv r_2-\ell \mod r_2$. In either case, we must have $t+\lfloor\frac{2}{r_1}\rfloor\geq 0$ and $\lfloor\frac{-2\ell}{r_2}\rfloor-t\geq 0$ for some integer $t$, which forces $r_1=2$ and either $H\simeq L_a^{-1}\otimes L_b^\ell$ or $H\simeq L_a\otimes L_b^{-\ell}$. The second case can be excluded since $\deg(H)>0$, and the first case gives a Higgs bundle of type (3), so every $L_a\otimes L_b^{-\ell}$-twisted Higgs bundle is of type (3).
    \end{proof}

    \subsection{The Generic Case}
    When $X$ is a generic Hopf surface, $H^0(X,\mc{T}_X\otimes \lambda^{-1})\neq 0$ and $\deg(\lambda)>0$ imply $\lambda\simeq L_a$, $\lambda\simeq L_b$, or $\lambda \simeq L_a\otimes (L_b)^{-\ell}$ for some integer $\ell$ with $0<\ell<\frac{\log|a|}{\log|b|}$.

    \begin{proposition}
        Let $X$ be a generic Hopf surface given by the contraction $(x,y)\mapsto (ax,by)$ with $|a|\leq |b|$. 
        \begin{enumerate}
            \item The nowhere-vanishing $L_a$-twisted Higgs bundles of the form $(\mc{O}_X\oplus H,\phi)$, up to conjugation by an automorphism of $\mc{O}_X\oplus H$, are as follows:
        \begin{enumerate}
            \item $H\simeq L_a^{-1}$ and $$\phi=\begin{pmatrix}
                0 & tx^2\\ 1 & 0
            \end{pmatrix}$$ for some $t \in \C$. The $\phi$-invariant sub-line bundles are $L_a^{-1}$ under the inclusions $(\pm x\sqrt{t}, 1).$
            \item $H\simeq L_a\otimes L_b^{-k}$ with $k\geq \frac{\log{|a|}}{\log{|b|}}$ and $$\phi=\begin{pmatrix}
                sx & y^k\\0 & -sx
            \end{pmatrix}$$ for some $a \in \C^*$. The $\phi$-invariant sub-line bundles are the first direct summand $\mc{O}_X$ and $(L_b)^{-k}$ under the inclusion $(y^k,-2sx).$
            \item $H\simeq L_b^k\otimes L_a^{-1}$ with $0<k< \frac{\log{|a|}}{\log{|b|}}$ and $$\phi=\begin{pmatrix}
                sx & 0\\y^k & -sx
            \end{pmatrix}$$ for some $s\in \C^*$. The $\phi$-invariant sub-line bundles are the second direct summand $H$ and $(L_a)^{-1}$ under the inclusion $(2sx,y^k)$.
        \end{enumerate}
        \item The nowhere-vanishing $L_b$-twisted Higgs bundles of the form $(\mc{O}_X\oplus H, \phi)$, up to conjugation by an automorphism of $\mc{O}_X\oplus H$, are as follows: 
        \begin{enumerate}
            \item $H\simeq L_b^{-1}$ and $$\phi=\begin{pmatrix}
                0 & ty^2\\1 & 0
            \end{pmatrix}$$ for some $t \in \C$. The $\phi$-invariant sub-line bundles are $L_b^{-1}$ under the inclusions $\pm y\sqrt{t}, 1)$.
            \item $H\simeq L_b\otimes L_a^{-k}$ with $k>0$ and $$\phi=\begin{pmatrix}
                sy & x^k\\
                0 & -sy
            \end{pmatrix}$$ for some $s \in \C^*$. The $\phi$-invariant line bundles are the first direct summand $\mc{O}_X$ and $(L_a)^{-k}$ under the inclusion $(x^k, -2sy)$.
        \end{enumerate}
        \item If $\ell$ is an integer with $0<\ell<\frac{\log{|a|}}{\log{|b|}}$, the unique $\lambda=L_a\otimes L_b^{-\ell}$-twisted Higgs bundle is $$\left(\mc{O}_X\oplus (L_b^\ell\otimes L_a^{-1}), \begin{pmatrix}
            0 & 0\\1 & 0
        \end{pmatrix}\right).$$ This Higgs bundle has $\lambda^{-1}$ as its only $\phi$-invariant sub-line bundle.
        \end{enumerate}
    \end{proposition}
    \begin{proof}
        First suppose that $(\mc{O}_X\oplus H, \phi)$ is a nowhere-vanishing $L_a$-twisted Higgs bundle. Then either $H\simeq L_a^{-1}$ and the $\mc{O}_X$-component of $\phi$ is non-zero, or one of $L_a\otimes H$ and $L_a\otimes H^{-1}$ has the form $L_b^k$ for some $k>0$. Since we assume $\deg(H)\leq 0$, the possible choices for $H$ are then $H\simeq L_a^{-1}$, $H\simeq L_a\otimes L_b^{-k}$ for some $k\geq \frac{\log{|a|}}{\log{|b|}}$, or $H\simeq L_b^k\otimes L_a^{-1}$ for some $0<k<\frac{\log{|a|}}{\log{|b|}}.$

        If $H\simeq L_a^{-1}$, then $\sEnd_0(E)\otimes L_a\simeq \mc{O}_X\oplus L_a\oplus L_a^2$, so $$\phi=\begin{pmatrix}
            sx & tx^2\\u & -sx
        \end{pmatrix}$$ for some $a,b\in \C$ and $c \in \C^*$. Conjugating by the automorphism $$\begin{pmatrix}
            u & -tx\\0 & 1
        \end{pmatrix}$$ gives $\phi$ of type (1a).

        If $H \simeq L_a\otimes L_b^{-k}$, then $\sEnd_0(E)\otimes L_a\simeq L_a^2\otimes L_b^{-k}\oplus L_a\oplus L_b^k$, so $$\phi=\begin{pmatrix}
            sx & ty^k\\0 & -sx
        \end{pmatrix}$$ for some $s,t\in \C^*$. Conjugating by the automorphism $$\begin{pmatrix}
            1 & 0\\0 & t
        \end{pmatrix}$$ gives $\phi$ of type (1b).

        If $H\simeq L_b^k\otimes L_a^{-1}$, then $\sEnd_0(E)\otimes L_a\simeq L_b^k\oplus L_a\oplus L_a^2\otimes L_b^{-k}$, so $$\phi=\begin{pmatrix}
            sx &0\\ uy^k & -sx
        \end{pmatrix}$$ for some $s,u\in \C^*$. Conjugating by the automorphism $$\begin{pmatrix}
            u & 0\\0 & 1
        \end{pmatrix}$$ gives $\phi$ of type (1c).

        The proof for $L_b$-twisted Higgs bundles is the same.

        If $(\mc{O}_X\oplus H, \phi)$ is a nowhere-vanishing $L_a\otimes L_b^{-\ell}$-twisted Higgs bundle for some integer $\ell$ with $0<\ell<\frac{\log{|a|}}{\log{|b|}},$ then $H\simeq L_a^j\otimes L_b^k$ for some integers $j,k$. Then $\sEnd_0(E)$ will have at most one component with a non-zero section, so we must have $H\simeq L_a^{-1}\otimes L_b^\ell$ and $$\phi=\begin{pmatrix}
            0 & 0\\u & 0
        \end{pmatrix}$$ for some $u \in \C^*$. Conjugating with the automorphism $$\begin{pmatrix}
            u & 0\\0 & 1
        \end{pmatrix}$$ gives $\phi$ of type (3).
    \end{proof}

     \subsection{The Exceptional Case}\label{exceptional-Higgs}
     When $X$ is an exceptional Hopf surface, the indecomposable extensions of line bundles admit a wider variety of Higgs fields than in the diagonal case. In order to compute the Higgs fields for these bundles, we note that every indecomposable bundle on the Hopf surface given by contraction $(x,y)\mapsto (b^rx+y^r, b y)$ can be written as a non-split extension in two distinct ways. Suppose that $E$ is the unique non-split extension of $(L_b)^{-\ell-rk}$ by $\mc{O}_X$ for integers $k,\ell$ with $k\geq 0$, $0\leq \ell<r$. Then $E$ also appears as a non-split extension of $(L_b)^{r-\ell}$ by $(L_b)^{-r(k+1)}$. Writing the two corresponding exact sequences $$\begin{tikzcd}
             0 \arrow[r] & \mc{O}_X\arrow{r}{\iota_1} & E \arrow{r}{\rho_1} & (L_b)^{-rk-\ell}\arrow[r] & 0,\\
             0 \arrow[r] & (L_b)^{-r(k+1)}\arrow{r}{\iota_2} & E \arrow{r}{\rho_2} & (L_b)^{r-\ell} \arrow[r] & 0,
         \end{tikzcd}$$
         
         We can construct some nowhere-vanishing Higgs fields \begin{align*}
             \phi_{11}&:=\left(\iota_1\otimes \id_{(L_b)^{r-\ell}}\right)\circ \rho_2,\\
             \phi_{12}&:=\left(\iota_1\otimes \id_{(L_b)^{-rk-\ell}}\right)\circ \rho_1,\\
             \phi_{21}&:=\left(\iota_2\otimes \id_{(L_b)^{r(k+2)-\ell}}\right)\circ \rho_2,\\
             \phi_{22}&:=\left(\iota_2\otimes \id_{(L_b)^{r-\ell}}\right)\circ \rho_1,
         \end{align*}
         where $\phi_{11}$ and $\phi_{22}$ are twisted by $(L_b)^{r-\ell}$, $\phi_{12}$ is twisted by $(L_b)^{-rk-\ell}$, and $\phi_{21}$ is twisted by $(L_b)^{r(k+2)-\ell}$. Clearly $\phi_{12}$ and $\phi_{21}$ are both nilpotent since $\rho_j\circ\iota_j=0$. By the fact that $\iota_1$ and $\iota_2$ have non-isomorphic domains and that they both have torsion-free quotient, we can conclude that $\rho_i\circ \iota_j\neq 0$ when $i\neq j$. Thus both $\phi_{11}$ and $\phi_{22}$ have rank-1 kernel and non-zero trace, and $\phi_{11}\circ\phi_{22}=\phi_{22}\circ\phi_{11}=0$. It also implies that $(\phi_{12}\otimes \id_{(L_b)^{r(k+2)-\ell}}\circ\phi_{21}\neq 0$ and $(\phi_{21}\otimes \id_{(L_b)^{-rk-\ell}})\circ \phi_{12}\neq 0$. Without loss of generality, we can assume that $\rho_1$ and $\rho_2$ were chosen so that $\tr(\phi_{11})=\tr(\phi_{22})=y^{r-\ell}$. Then $\phi_{11}-\phi_{22}$ gives a trace-free $(L_b)^{r-\ell}$-twisted Higgs field on $E$. We can conclude that for any $\lambda \in \Div$, there is a local frame for $\sEnd_0(E)\otimes \lambda$ given by $\{s_{0}(\phi_{11}-\phi_{22}), s_{12}(\phi_{12}),s_{21}(\phi_{21})\}$, where $s_{0}$ is a meromorphic section of $\lambda\otimes (L_b)^{\ell-r}$, $s_{12}$ is a meromorphic section of $\lambda\otimes (L_b)^{rk+\ell}$, and $s_{21}$ is a meromorphic section of $\lambda\otimes (L_b)^{\ell-r(k+2)}$. 

         \begin{lemma}
             If $X$ is an exceptional Hopf surface and $E$ is the unique non-split extension of $(L_b)^{-rk-\ell}$ by $\mc{O}_X$ for non-negative integers $k,\ell$ with $\ell<r$, the meromorphic sections $\phi_{11}$, $\phi_{22}$, $\phi_{12}$, and $\phi_{21}$ of $\sEnd(E)$ satisfy the relations \begin{align*}
                 \phi_{11}+\phi_{22}&=y^{r-\ell}\id_E,\\
                 \phi_{11}\circ \phi_{22}=\phi_{22}\circ \phi_{11}=\phi_{11}\circ\phi_{21}&=\phi_{21}\circ\phi_{22}=\phi_{12}\circ\phi_{11}=\phi_{22}\circ\phi_{12}=0,\\
                 \phi_{11}\circ \phi_{12}&=\phi_{12}\circ \phi_{22}=y^{r-\ell}\phi_{12},\\
                 \phi_{22}\circ\phi_{21}&=\phi_{21}\circ\phi_{11}=y^{r-\ell}\phi_{21}.
             \end{align*}
         \end{lemma}
         \begin{proof}
             The fact that $\phi_{11}\circ \phi_{22}=\phi_{22}\circ \phi_{11}=\phi_{11}\circ\phi_{21}=\phi_{21}\circ\phi_{22}=\phi_{12}\circ\phi_{11}=\phi_{22}\circ\phi_{12}=0$ follows immediately from the definitions of the Higgs fields and the fact that $\rho_j\circ\iota_j=0$ for $j=1,2$. The fact that $\phi_{11}$ and $\phi_{22}$ have trace $y^{r-\ell}$ and zero determinant implies that $\phi_{11}^2=y^{r-\ell}\phi_{11}$ and $\phi_{22}^2=y^{r-\ell}\phi_{22}$, so $(\phi_{11}+\phi_{22})^2=\phi_{11}^2+\phi_{22}^2=y^{r-\ell}(\phi_{11}+\phi_{22})$. The fact that $\tr(\phi_{11}+\phi_{22})=2y^{r-\ell}$ then forces $\phi_{11}+\phi_{22}=y^{r-\ell}\id_E$.

             Using the fact that $\phi_{22}\circ\phi_{12}=0$ and $\phi_{11}+\phi_{22}=y^{r-\ell}\id_E$ we get $\phi_{11}\circ\phi_{12}=\phi_{11}\circ\phi_{12}+\phi_{22}\circ\phi_{12}=y^{r-\ell}\phi_{12}$. The computations for $\phi_{12}\circ \phi_{22}$, $\phi_{22}\circ\phi_{21}$, and $\phi_{21}\circ\phi_{11}$ work similarly.
         \end{proof}

     We now return to $\lambda$-twisted Higgs fields where $H^0(X,\mc{T}_X\otimes \lambda^{-1})\neq 0$. This restriction forces $\lambda\simeq (L_b)^\ell$ for some $\ell\leq r$.

     \begin{proposition}
         Let $X$ be the exceptional Hopf surface given by the contraction $(x,y)\mapsto (b^rx+y^r, by)$. Set $\lambda\simeq (L_b)^\ell$ for some $\ell \in \Z$ with $\ell\leq r$.
         \begin{enumerate}
             \item When $1\leq \ell\leq r$, the $\lambda$-twisted Higgs bundles $(E,\phi)$, up to conjugation with an automorphism of $E$, are one of:
             \begin{enumerate}
                 \item $E\simeq \mc{O}_X\oplus \lambda^{-1}$ with $$\phi=\begin{pmatrix}
                 0 & by^{2\ell}\\1 & 0
                \end{pmatrix}$$
                for some $b \in \C$. The $\phi$-invariant sub-line bundles are $\lambda^{-1}$ under the inclusion $(\pm y^\ell \sqrt{b}, 1).$
                \item $E$ is the unique non-split extension of $(L_b)^{\ell-rk}$ by $\mc{O}_X$ for some integer $k\geq 1$, and $$\phi=\alpha(\phi_{11}-\phi_{22})$$ for some $\alpha \in \C^*$. The $\phi$-invariant sub-line bundles are given by the inclusions $\iota_1:\mc{O}_X\hookrightarrow E$ and $\iota_2: (L_b)^{-r(k+1)} \hookrightarrow E$.
             \end{enumerate}
             \item When $\ell=0$, the $\lambda$-twisted Higgs bundles $(E,\phi)$, up to conjugation with an automorphism of $E$, are one of:
             \begin{enumerate}
                 \item $E\simeq \mc{O}_X\oplus \mc{O}_X$ and $$\phi=\begin{pmatrix}
             0 & 1\\0 & 0
         \end{pmatrix}.$$ The second direct summand is the unique $\phi$-invariant sub-line bundle.
         \item $E \simeq \mc{O}_X\oplus H$ for some $H\in \Pic(X)$ with $$\phi=\begin{pmatrix}
             a & 0\\0 & -a
         \end{pmatrix}$$ for some $a \in \C^*$. Both direct summands are $\phi$-invariant.
         \item $E$ is the unique non-split extension of $\mc{O}_X$ by itself, and $\phi=\beta\phi_{12}$ for some $\beta \in \C^*$. The unique $\phi$-invariant sub-line bundle is given by $\iota_1:\mc{O}_X\hookrightarrow E$.
             \end{enumerate}
             \item For $\ell<0$, the $\lambda$-twisted Higgs bundles $(E,\phi)$, up to conjugation with an automorphism of $E$, are one of:
             \begin{enumerate}
                 \item $E\simeq \mc{O}_X\oplus \lambda$ and $$\phi=\begin{pmatrix}
             0 & 1\\0 & 0
         \end{pmatrix}.$$
         Here the second direct summand is the only $\phi$-invariant sub-line bundle.
         \item $E$ is the unique non-split extension of $\lambda$ by $\mc{O}_X$, and $\phi=\beta\phi_{12}$ for some $\beta\in \C^*$. The unique $\phi$-invariant sub-line bundle is given by $\iota_1:\mc{O}_X\hookrightarrow E$.
             \end{enumerate}
         \end{enumerate}
     \end{proposition}
     \begin{proof}
         We begin with the decomposable bundles $E\simeq \mc{O}_X\oplus H$ for $H \in \Pic(X)$ with $\deg(H)\leq 0$. Since $X$ has a unique divisor, the only way to have $\phi\in H^0(X,\sEnd_0(E)\otimes \lambda)=H^0(X,\lambda)\oplus H^0(X,\lambda\oplus H)\oplus H^0(X,\lambda\oplus H^{-1})$ not vanish on the divisor is if $H\simeq \lambda$, $H\simeq \lambda^{-1}$, or $\lambda\simeq \mc{O}_X$. Using analogous computations to the generic case, the case $H\simeq \lambda^{-1}$ gives type (1a), the case $H\simeq \lambda$ gives type $(2a)$ or $(3a)$, and the case $\lambda\simeq \mc{O}_X$ gives type $(2a)$ or $(2b)$.

         Now consider the case where $E$ is the indecomposable extension of $(L_b)^{-rk-m}$ by $\mc{O}_X$ for integers $k,m$ with $k\geq 0$ and $0\leq m<r$. Suppose that $\phi$ is a trace-free $\lambda$-twisted Higgs field on $E$ for $\lambda\simeq (L_b)^\ell$ which does not vanish on any divisor. Then there are sections $s_{0}\in H^0(X,\lambda\otimes (L_b)^{m-r}), s_{12}\in H^0(X,\lambda\otimes (L_b)^{rk+m}), s_{21}\in H^0(X,\lambda\otimes (L_b)^{m-r(k+2)})$ so that $\phi=s_{0}\phi_{11}+s_{12}\phi_{12}+s_{21}\phi_{21}$, and such that $V(s_{0})\cap V(s_{12})\cap V(s_{21})=\emptyset$. However, since $X$ has a unique divisor, $V(s_{0})\cap V(s_{12})\cap V(s_{21})=\emptyset$ can only occur if $\lambda\in \{(L_b)^{r-m}, (L_b)^{-rk-m}, (L_b)^{r(k+2)-m}\}$. The restriction that $\ell\leq r$ rules out the third case, so we only need to consider when $\ell=r-m$ or $\ell=-rk-m$. In the case $\ell=r-m$, we get $\phi=\alpha (\phi_{11}-\phi_{22})+\beta y^{r(k+1)}\phi_{12}$ for $\alpha \in \C^*$ and $\beta \in \C$. Conjugating with the automorphism $\id_E+\frac{\beta y^{rk+m}}{2\alpha}\phi_{12}$ gives a Higgs field of type $(1b)$. When $\ell=-rk-m$, we get $\phi=\beta \phi_{12}$ for some $\beta \in \C^*$. This gives a Higgs field of type $(2c)$ or $(3b)$ depending on whether $rk+m=0$.
     \end{proof}

\section{Co-Higgs bundles with $c_2>0$}\label{co-Higgs-non-flat}

From Proposition \ref{non-filtrable} we know that every rank-2 co-Higgs bundle on a Hopf surface will be filtrable, so the underlying bundle will always admit a Serre-type exact sequence. Our strategy for studying co-Higgs bundles with positive second Chern class will vary based on the structure of this exact sequence.

\begin{proposition}\label{not-in-fibre}
    Let $E$ be a rank-2 filtrable vector bundle on a Hopf surface $X$, and let \begin{equation}\label{Serre}\begin{tikzcd}
        0 \arrow[r] & L_1 \arrow{r}{\iota} & E \arrow{r}{p} & L_2\otimes I_Z \arrow[r] & 0
    \end{tikzcd}\end{equation}
    be a Serre exact sequence for $E$ with $L_1,L_2 \in \Pic(X)$ and $Z$ a zero-dimensional subspace. If the support of $Z$ is not contained in a divisor of $X$, then for any $H \in \Div$, $H^0(X,\sEnd_0E\otimes H)\simeq H^0(X,L_1\otimes L_2^{-1} \otimes H)$. In particular, any section $\phi \in H^0(X,\sEnd_0E\otimes H)$ is nilpotent.
\end{proposition}
\begin{proof}
    First we take the tensor product of the Serre exact sequence with $H\otimes E^\vee$ to get
    $$\begin{tikzcd}
        0 \arrow[r] & L_1\otimes H\otimes E^\vee \arrow[r] & \sEnd E\otimes H\arrow[r] & L_2\otimes H\otimes E^\vee \otimes I_Z \arrow[r] & 0.
    \end{tikzcd}$$
    If we rewrite this using the fact that $E^\vee\simeq E\otimes \det(E)^{-1}\simeq E\otimes L_1^{-1}\otimes L_2^{-1}$ and take cohomology, we get the left exact sequence 
    $$\begin{tikzcd}
        0 \arrow[r] & H^0(X,L_2^{-1}\otimes H\otimes E) \arrow{r}{\iota_*} & H^0(X,\sEnd E\otimes H) \arrow{r}{p_*} & H^0(X,L_1^{-1}\otimes H\otimes E\otimes I_Z).
    \end{tikzcd}$$

    Note that the restriction of $p_*$ to the identity component $\{s\id:s \in H^0(X,H)\}\subseteq H^0(X,\sEnd E\otimes H)$ is necessarily injective by surjectivity of $p$. In particular, this implies that $h^0(X,L_1^{-1}\otimes H\otimes E\otimes I_Z)\geq h^0(X,H).$

    As an ideal sheaf, $I_Z$ fits naturally into the exact sequence $$\begin{tikzcd}
        0 \arrow[r] & I_Z \arrow[r] & \mc{O}_X \arrow[r] & \mc{O}_Z \arrow[r] & 0.
    \end{tikzcd}$$
    If $F$ is any locally free sheaf, tensoring with $F$ and taking cohomology gives the left exact sequence \begin{equation}\label{ideal-exact}\begin{tikzcd}
        0 \arrow[r] & H^0(X,F\otimes I_Z) \arrow[r] & H^0(X,F) \arrow{r}{\cdot|_Z} & H^0(Z,F|_Z),
    \end{tikzcd}\end{equation}
    where the rightmost map is evaluation at $Z$, so $H^0(X,F\otimes I_Z)=\{\sigma\in H^0(X,F):Z\subseteq V(\sigma)\}.$

    Tensoring the Serre exact sequence with $L_1^{-1}\otimes H$ and taking cohomology gives $$\begin{tikzcd}
        0 \arrow[r] & H^0(X,H)\arrow[r] & H^0(X,L_1^{-1}\otimes H\otimes E)\arrow[r] & H^0(X,L_1^{-1}\otimes L_2\otimes H\otimes I_Z).
    \end{tikzcd}$$
    Since the support of $Z$ is not contained in any divisor, we must have $H^0(X,L_1^{-1}\otimes L_2\otimes H\otimes I_Z)=0$ by invoking \eqref{ideal-exact} with $F=L_1^{-1}\otimes L_2\otimes H$.

    Applying \eqref{ideal-exact} again with $F=L_1^{-1}\otimes H\otimes E$, we see that $$h^0(X,L_1^{-1}\otimes H\otimes E\otimes I_Z)\leq h^0(X,L_1^{-1}\otimes H\otimes E)=h^0(X,H),$$ meaning that the restriction of $p_*$ to the identity component is an isomorphism, and therefore there is an isomorphism between $H^0(X,L_2^{-1}\otimes H\otimes E)$ and $H^0(X,\sEnd_0E\otimes H)$.

    Finally, we can tensor the Serre exact sequence by $L_2^{-1}\otimes H$ and take cohomology to get $$\begin{tikzcd}
        0 \arrow[r] & H^0(X,L_1\otimes L_2^{-1}\otimes H) \arrow[r] & H^0(X,L_2^{-1}\otimes H\otimes E) \arrow[r] & H^0(X,H\otimes I_Z)=0,
    \end{tikzcd}$$
    so that $H^0(X,L_1\otimes L_2^{-1}\otimes H)\simeq H^0(X,\sEnd_0E\otimes H)$. Given a section $s \in H^0(X,L_1\otimes L_2^{-1}\otimes H)$ the corresponding element of $H^0(X,\sEnd_0E\otimes H)$ can be given explicitly. Let $\varphi_s:L_2\otimes I_Z\to L_1\otimes H$ be the map given by composing the inclusion of $L_2\otimes I_Z$ in $L_2$ with multiplication by $s$. Then $\phi_s=\iota\circ \varphi_s\circ p$ is naturally a section of $H^0(\sEnd_0E\otimes H)$ which is zero if and only if $s$ is.

    Note that $\phi_s\circ \phi_s=\iota\circ\varphi_s\circ (p\circ \iota)\circ \varphi_s\circ p=0$ by exactness of \eqref{Serre}, so this Higgs is always nilpotent.

\end{proof}
\begin{remark}
    If $L_1,L_2 \in \Pic(X)$ are such that $H^0(X,L_1\otimes L_2^{-1}\otimes \lambda)\neq 0$ for some $\lambda \in \Pic(X)$ with $H^0(X,\mc{T}_X\otimes \lambda^{-1})\neq 0$, then a bundle of type \eqref{Serre} is guaranteed to exist. To see this, note that $\deg(\lambda)<-\deg(K_X)$ and $\deg(L_1\otimes L_2^{-1})> \deg(K_X)$ by the fact that $H^0(X,L_1\otimes L_2^{-1}\otimes \lambda)$ and $H^0(X,\mc{T}_X\otimes \lambda^{-1})$ are both non-zero. The obstruction to the existence of locally-free sheaves satisfying \eqref{Serre} is when $H^2(X,L_1\otimes L_2^{-1})\neq0$. However in this case, $H^2(X,L_1\otimes L_2)\simeq H^0(X,K_X\otimes L_1^{-1}\otimes L_2)=0$ as $\deg(K_X\otimes L_1^{-1}\otimes L_2)<0$. Thus there are vector bundles satisfying \eqref{Serre} which admit non-trivial co-Higgs bundles.
\end{remark}
If instead $E$ is filtrable of the form $$\begin{tikzcd}
    0 \arrow[r] & L_1 \arrow[r] & E \arrow[r] & L_2\otimes I_Z \arrow[r] & 0
\end{tikzcd}$$ and there is an effective divisor $D$ so that the support of $Z$ is contained in $D$, then Higgs fields on $E$ can be computed via elementary modifications.
\begin{definition}
    Let $j:D\hookrightarrow X$ be the inclusion of a prime divisor, and take $\lambda \in \Pic(D)$. 
    An \emph{elementary modification} of a rank-2 bundle $E$ is given by $\ker(p_j)$, where $p_j:E\to j_*\lambda$ is a surjective morphism of sheaves. 
    If $E|_D$ is slope-unstable, the \emph{allowable elementary modification} of $E$ at $D$ is the elementary modification corresponding to the unique surjective map $E|_D\to L$ such that $L\in \Pic(D)$ has $\mu(L)<\mu(E|_D)$, where $\mu$ is the slope function on $D$.
    In this case we say that $E$ has a \emph{jump} at $D$. If an allowable elementary modification of $E$ at $D$ can be performed $\ell$ times before the resulting bundle $E_0$ is such that $E_0|_D$ is semi-stable, we say that the jump of $E$ at $D$ has \emph{length} $\ell$.
    
\end{definition}
Since on a Hopf surface all divisors have trivial self-intersection, the determinant and second Chern class of an elementary modification $$\begin{tikzcd}
    0\arrow[r] & E'\arrow[r] & E \arrow[r] &j_*\lambda \arrow[r] & 0
\end{tikzcd}$$
are given by \begin{align*}
    \det(E')&\simeq \det(E)\otimes \mc{O}_X(-D),\\
    c_2(E)&=c_2(E)+j_*c_1(\lambda).
\end{align*}

\begin{lemma}\label{ideal-fibre}
    Let $E$ be a vector bundle on a Hopf surface $X$ satisfying the exact sequence $$\begin{tikzcd}
    0 \arrow[r] & L_1 \arrow[r] & E \arrow[r] & L_2\otimes I_Z \arrow[r] & 0
\end{tikzcd}$$ for some $L_1,L_2 \in \Pic(X)$ and $Z$ a zero-dimensional subspace, and let $D$ be an integral divisor on $X$ with $Z'=Z\cap D\neq \emptyset$ the corresponding divisor on $D$. Then $E|_D$ is unstable.
\end{lemma}
\begin{proof}
    Pulling back the exact sequence for $E$ by the inclusion of $D$ gives a right-exact sequence $$\begin{tikzcd}
        L_1|_D \arrow[r] & E|_D \arrow[r] & (L_2\otimes I_Z)|_D\arrow[r] & 0.
    \end{tikzcd}$$ Since $Z$ is a local complete intersection, on the level of stalks there are $f,g \in \mc{O}_{X,x}$ so that $I_{Z,x}=(f,g)$ and $\mc{O}_{D,x}=\mc{O}_{X,x}/(f),$ so $I_{Z,x}\otimes \mc{O}_{D,x}\simeq (g|_D)\oplus \mc{O}_{Z,x}$. Returning to the global picture, we see $I_Z|_D\simeq L_2|_D\otimes \mc{O}_D(-Z')\oplus \mc{O}_{Z'}$. Recall that since Hopf surfaces have trivial Néron-Severi group, the restriction of a vector bundle to $D$ has trivial first Chern class. Now $\mu(E|_D)=0$ and $\mu(L_2|_D\otimes \mc{O}_D(-Z'))=-\deg(Z')<0$, so $E|_D$ must be unstable as it admits a surjection to a line bundle with negative slope.
\end{proof}
\begin{proposition}\label{elementary-modification-sequence}
    Let $X$ be a Hopf surface, and $E$ a filtrable rank-2 vector bundle on $X$. If $E$ has an exact sequence $$\begin{tikzcd}
        0 \arrow[r] & L_1 \arrow[r] & E \arrow[r] & L_2\otimes I_Z \arrow[r] &0
    \end{tikzcd}$$ for some $L_1,L_2 \in \Pic(X)$ and $Z$ a zero-dimensional subspace with support contained in some divisor $D \subset X$, then there is a sequence of allowable elementary modifications taking $E$ to an extension of line bundles.
\end{proposition}
\begin{proof}
    We prove the result via induction on the second Chern class. Since $NS(X)$ is trivial, $c_2(E)=c_2(I_Z)$, so if $c_2=0$ then $L_2\otimes I_Z$ is a line bundle. Now suppose that the result holds whenever $c_2(I_Z)<k$ for some $k \in \Z$ and that $c_2(E)=k$. Since the support of $Z$ is contained in a divisor of $X$, there is a prime divisor $j:D_0\hookrightarrow E$ so that $Z\cap D_0\neq \emptyset$. By Lemma \ref{ideal-fibre}, $E$ admits an allowable elementary modification $$\begin{tikzcd}
        0 \arrow[r] & E'\arrow[r] & E \arrow[r] & j_*\lambda \arrow[r] & 0
    \end{tikzcd}$$
    with $\lambda\in \Pic(D_0)$ the line bundle of minimal degree admitting a surjection from $E|_{D_0}$. Using the Chern class formula for an elementary modification \cite[Lemma 2.16]{Friedman98}, $$c_2(E')=c_2(E)-c_1(E)\cdot D_0 +j_*c_1(\lambda)=k+j_*c_1(\lambda)<k$$ since $c_1(E)=0, c_2(E)=k$, and $c_1(\lambda)<0$. Note that $E'$ is filtrable fitting into the exact sequence \begin{equation}\label{allowable-exact}\begin{tikzcd}
        0 \arrow[r] & L_1 \arrow[r] & E'\arrow[r] L_2\otimes \mc{O}_X(-D_0)\otimes I_{Z'} \arrow[r] &0
    \end{tikzcd}\end{equation} for some zero-dimensional subspace $Z'$. Since $E'$ is isomorphic to $E$ away from $D_0$, the support of $Z'$ must also be contained in the divisor $D$ containing the support of $Z$. By our induction hypothesis, $E'$ admits a sequence of allowable elementary modifications taking it to an extension of line bundles, so adding  the elementary modification \eqref{allowable-exact} to the sequence gives the result for $E$.
\end{proof}

Now for every filtrable bundle, we can either compute Higgs fields directly or relate it to an extension of line bundles via an allowable elementary modification. By understanding the effect of an allowable elementary modification on Higgs fields, we will be able to fully describe $\lambda$-twisted Higgs bundles for $\lambda \in \Pic(X)$ with $H^0(X,\mc{T}_X\otimes \lambda^{-1})\neq0$.

\begin{lemma}\label{restrict-phi}
    Let $E$ be a rank-2 vector bundle on a Hopf surface $X$ with an allowable elementary modification \begin{equation}\label{allowable}\begin{tikzcd}0 \arrow[r] & E' \arrow{r}{\iota} & E \arrow[r] & j_*L\arrow[r] & 0\end{tikzcd}\end{equation}
    on a divisor $j:D\to X$ with $\deg(L)=d<0$.
    Then $E|_D\simeq L\oplus L'$ for some $L' \in \Pic^{-d}(D),$ and $E'|_D$ is an extension of $L'$ by $L$.
\end{lemma}
\begin{proof}
    First, $D$ is an elliptic curve, and therefore any unstable bundle on $D$ is a direct sum of line bundles. We also have that $\deg(E|_D)=0$ \cite[Corollary 1.5]{Teleman98}, so there must be a line bundle $L'\in \Pic^{-d}(D)$ so that $E|_D\simeq L\oplus L'$. Restricting the short-exact sequence \eqref{allowable} to $D$ gives a new exact sequence $$\begin{tikzcd}
        0 \arrow[r] & j^*\mathrm{Tor}_1^{\mc{O}_X}(\mc{O}_D,j_*L) \arrow[r] & E'|_D \arrow{r}{\iota|_D} & E|_D \arrow[r] & L \arrow[r] & 0.
    \end{tikzcd}$$
    Since $H^0(D,(L')^{-1}\otimes L)=0$, this gives a new short-exact sequence $$\begin{tikzcd}
        0 \arrow[r] & j^*\mathrm{Tor}_1^{\mc{O}_X}(\mc{O}_D,j_*L) \arrow[r] & E'|_D \arrow[r] & L'\arrow[r] & 0.
    \end{tikzcd}$$
    Using the flat resolution $$\begin{tikzcd}
        0 \arrow[r] & \mc{O}_X(-D)\arrow[r] & \mc{O}_X \arrow[r] & \mc{O}_D \arrow[r] & 0
    \end{tikzcd}$$
    of $\mc{O}_D$, we can easily check that $\mathrm{Tor}_1^{\mc{O}_X}(\mc{O}_D, j_*L)\simeq j_*L$, so that $E'|_D$ is an extension of $L'$ by $L$.
\end{proof}
\begin{proposition}\label{allowable-endomorphism}
    Let $E$ be a rank-2 vector bundle on a Hopf surface $X$ with algebraic dimension 1, and suppose that $E$ has an allowable elementary modification \begin{equation}\label{2.5-1}\begin{tikzcd}
        0\arrow[r] & E'\arrow{r}{\iota} & E \arrow[r] & j_*L\arrow[r] & 0
    \end{tikzcd}\end{equation}
    for $j:D\to X$ the inclusion of a prime divisor. 
    Let \begin{equation}\label{2.4-2}\begin{tikzcd}
        0 \arrow[r] &E\otimes \mc{O}_X(-D) \arrow{r}{\iota'} & E' \arrow[r] & j_*(L')\arrow[r] & 0,
    \end{tikzcd}\end{equation}
    be the complementary elementary modification. The map from $\sEnd(E')$ to $\sEnd(E)\otimes \mc{O}_X(D)$ given by $\phi\mapsto \iota\circ \phi\circ \iota'$ induces an isomorphism between $\pi_*(\sEnd_0(E')\otimes \lambda)$ and $\pi_*(\sEnd_0(E)\otimes \lambda\otimes \mc{O}_X(D))$ for any $\lambda \in \Div$, and for any $\phi\in H^0(X,\sEnd(E')\otimes \lambda)$, $(\iota\circ\phi\circ\iota')|_D=0$ if and only if $\phi|_D$ is a multiple of the identity.
\end{proposition}
\begin{remark}
    Using the projection formula, it an immediate consequence of the above proposition that $H^0(X,\sEnd_0(E')\otimes \lambda)\simeq H^0(X,\sEnd_0(E)\otimes \lambda\otimes \mc{O}_X(D))$ for any $\lambda \in \Div$.
\end{remark}
\begin{proof}
    
    Note that $\iota$ and $\iota'$ are both isomorphisms on a dense open set, so $\pi_*\sEnd_0(E')$ and $\pi_*\sEnd_0(E)$ have the same rank.

     We prove this result by induction on the length of the jump of $E$ at $D$. 
     We start with the base case where an allowable elementary modification at $D$ can be performed exactly once, meaning that $E'|_D$ is semi-stable. 
     In this case, we first look at the left exact sequence $$\begin{tikzcd}
        0 \arrow[r] & \pi_*(\sEnd(E')\otimes \lambda) \arrow{r}{\iota_*} & \pi_*(\sHom(E', E)\otimes \lambda) \arrow[r] & \pi_*\sHom(E', j_*L\otimes \lambda)
    \end{tikzcd}$$
    generated by tensoring the sequence \eqref{2.5-1} with $E'^*\otimes \lambda$ and taking the pushforward by $\pi$. As $\pi(D)$ is a point, we have an isomorphism $\pi_*\sHom(E', j_*L\otimes)\simeq \mc{O}_{\pi(D)}\otimes H^0(D, E'|_D^\vee \otimes L)$ (recall that $\lambda|_D=\mc{O}_D$ for any $\lambda\in \Div$). 
    Since $E'$ does not admit an allowable elementary modification at $D$ and $\deg(L)<0$, $H^0(D, E'|_D^\vee \otimes L)=0$, so that pushing forward by $\iota$ gives an isomorphism $\pi_*(\sEnd(E')\otimes \lambda)\simeq \pi_*(\sHom(E',E)\otimes \lambda).$

    We now look at the exact sequence $$\begin{tikzcd}
        0 \arrow[r] & \pi_*(\sHom(E',E)\otimes \lambda) \arrow{r}{\iota'^*} & \pi_*(\sEnd(E)\otimes\lambda\otimes \mc{O}_{X}(D)) \arrow[r] & \pi_*\sExt_X^1(j_*L',E\otimes \lambda)
    \end{tikzcd}$$
    obtained by applying $\pi_*\sHom(\cdot, E\otimes \lambda)$ to the exact sequence \eqref{2.4-2}. As shown in \cite[p.42]{Friedman98}, $$\pi_*\sExt_X^1(j_*L',E)\simeq \mc{O}_{\pi(D)}\otimes H^0(D,E|_D\otimes L'^{-1})=\mc{O}_{\pi(D)}\otimes H^0(D, \mc{O}_D\oplus (L\otimes L'^{-1}))=\mc{O}_{\pi(D)}.$$
    Substituting with the isomorphism from pushing forward by $\iota$ gives the left-exact sequence $$\begin{tikzcd}
        0 \arrow[r] & \pi_*(\sEnd(E')\otimes \lambda)\arrow{r}{\iota_*\circ\iota'^*} & \pi_*(\sEnd(E)\otimes \lambda\otimes \mc{O}_{X}(D))\arrow{r}{p_D} & \mc{O}_{\pi(D)}.
    \end{tikzcd}$$
    Note that for any local section $\phi \in H^0(U,\sEnd(E'))$, $(\iota\circ\phi\circ\iota')|_D$ is nilpotent, since $\ker(\iota|_D)=\im(\iota'|_D)\simeq L$ and $\im(\iota|_D)=\ker(\iota'|_D)\simeq L'$ by Lemma \ref{restrict-phi}. Also, since $\iota\circ\iota'$ is exactly $s\id_E$ for some $s\in H^0(X,\mc{O}_X(D))$ which vanishes on $D$, the map $p_D$ can be interpreted as $\psi\mapsto \frac{1}{2}\tr \psi|_{\pi(D)}$. We can now conclude that the above exact sequence splits into $$
       \begin{tikzcd} 0 \arrow[r] &\pi_*\lambda \arrow[r] &\pi_*(\lambda\otimes \mc{O}_{X}(D)) \arrow[r] & \mc{O}_{\pi(D)}\end{tikzcd}$$
       and
       $$\begin{tikzcd}
           0 \arrow[r] & \pi_*(\sEnd_0(E')\lambda) \arrow[r] & \pi_*(\sEnd_0(E)\otimes \lambda\otimes \mc{O}_{B}(D)) \arrow[r] & 0,
       \end{tikzcd}$$
       giving the desired isomorphism.
       
       It remains to show that $\phi|_D$ is a multiple of the identity if and only if $(\iota\circ\phi\circ\iota')|_D=0$. 
       One direction immediately follows from the fact that $\ker(\iota|_D)=\im(\iota'|_D)$, so assume that $(\iota\circ\phi\circ\iota')|_D=0$. 
       Then we must have $\im((\phi\circ\iota')|_D)\subseteq \ker(\iota|_D)=\im(\iota'|_D)$, so in particular $\im(\iota'|_D)\simeq L$ is $\phi|_D$ invariant. 
       But $\phi|_D$ is a section in $H^0(D,\sEnd(E'|_D))$ where $E'$ is a semi-stable bundle of degree zero. It is easy to check that any endomorphism of such an $E'$ admitting a negative-degree invariant subbundle with torsion-free quotient is a multiple of the identity.

       Now suppose that the result holds for bundles whose jump at $D$ has length $\ell$ for some $\ell\geq 1$, and let $E$ be a bundle whose jump at $D$ has length $\ell+1$. 
       Choose $\lambda\in \Div$ and $\phi\in H^0(X,\sEnd_0(E')\otimes\lambda)$ so that $\phi$ does not vanish identically on $D$. Suppose that $E'|_D\simeq M\oplus M'$ for some $M,M'\in \Pic(D)$ with $\deg(M)=-\deg(M')<0$. 
       By our assumption that the proposition holds for $E'$, necessarily there is a non-zero section $a \in H^0(D,M'\otimes M^{-1})$ so that $$\phi|_D=\begin{bmatrix}0 & 0\\a & 0\end{bmatrix}.$$ 
       By Lemma \ref{restrict-phi} there are sections $x\in H^0(D,M\otimes L^{-1})$ and $y\in H^0(D,M'\otimes L^{-1})$ with no common zeros so that \begin{align*}\iota'|_D=\begin{bmatrix}x & 0\\ y & 0\end{bmatrix}, && \iota|_D=\begin{bmatrix}
           0 & 0\\ -y & x
       \end{bmatrix},\end{align*} where we are using the isomorphisms $E'|_D\simeq M\oplus M'$ and $E|_D\simeq L\oplus L'$. 
       (Since $M\oplus M'$ and $L\oplus L'$ have the same determinant, $M\otimes L^{-1}\simeq L'\otimes (M')^{-1}$ and $M'\otimes L^{-1}\simeq L'\otimes M^{-1}$.) 
       Combining these representations gives 
       $$(\iota\circ\phi\circ\iota')|_D=\begin{bmatrix}0 & 0\\ -y & x\end{bmatrix}\begin{bmatrix}0 & 0\\ a & 0\end{bmatrix}\begin{bmatrix}x & 0\\y & 0\end{bmatrix}=\begin{bmatrix}0 & 0\\ ax^2 & 0\end{bmatrix},$$ 
       so that $\phi\mapsto \iota\circ\phi\circ\iota'$ vanishes on at most a finite set (the zero-set of $x$). 
       Choose $\lambda \in \Div$ and $\{\phi_1, \ldots, \phi_k\}\subset H^0(X,\sEnd_0(E')\otimes\lambda)$ so that $\{\phi_1,\ldots, \phi_k\}$ gives rise to a local frame for $\pi_*(\sEnd_0(E')\otimes \lambda)$ in a Zariski-neighbourhood of $\pi(D)$. 
       Then since $\phi\mapsto \iota\circ\phi\circ\iota'$ is an isomorphism away from $D$, $\{\iota\circ\phi_1\circ\iota', \ldots, \iota\circ\phi_k\circ\iota'\}$ give a local frame for $\pi_*(\sEnd_0(E)\otimes\lambda\otimes \mc{O}_{X}(D))$ on a non-empty Zariski-open subset. 
       Furthermore, any local section of$\pi_*(\sEnd_0(E)\otimes\lambda\otimes \mc{O}_{X}(D))$ will be given by a $K(B)$-linear combination of the $\{\iota\circ\phi_1\circ\iota', \ldots, \iota\circ\phi_k\circ\iota'\}$. 
       Note that there is no divisor on which a $\iota\circ\phi_i\circ\iota'$ vanishes but the corresponding $\phi_i$ does not, so $\phi\mapsto\iota\circ\phi\circ\iota'$ gives a genuine isomorphism between $\pi_*\sEnd_0(E')$ and $\pi_*\sEnd_0(E)\otimes \mc{O}_{B}(\pi(D))$.

       Finally, note that for any line bundle $\lambda\in \Pic(B)$ and $\phi\in H^0(X, \sEnd(E')\otimes \pi^*\lambda)$, $\phi$ can be uniquely written as $\phi=s\id_{E'}+\phi_0$, where $s \in H^0(X,\pi^*\lambda)$ and $\phi_0\in H^0(X,\sEnd_0(E')\otimes \pi^*\lambda).$ As shown in the base case, $(\iota\circ s\id_{E'}\circ \iota')|_D=0$ since $\ker(\iota|_D)=\im(\iota'|_D)$, so $(\iota\circ\phi\circ\iota')|_D=(\iota\circ\phi_0\circ\iota')|_D$. If $\phi_0$ vanishes identically on $D$, then $\phi_0|_D=0\id_{E'}$, and we already showed that $(\iota\circ\phi_0\circ\iota')|_D\neq 0$ if $\phi_0$ does not vanish identically on $D$.
\end{proof}
For the Hopf surfaces with algebraic dimension 0, we will use a slightly different technique to address elementary modifications.
\begin{lemma}\label{alg-zero-nilpotent}
    Let $X$ be a Hopf surface with algebraic dimension 0 (so of generic or exceptional type), and let $(E,\phi)$ be a rank-2 trace-free $\lambda$-twisted Higgs bundle for some $\lambda \in \Div$ such that $E$ admits an allowable elementary modification 
    $$\begin{tikzcd}
        0 \arrow[r] & E' \arrow[r] & E \arrow[r] & j_*L\arrow[r] & 0
    \end{tikzcd}$$ for an irreducible divisor $j:D\hookrightarrow X$ and a line bundle $L \in \Pic(D)$ with $\deg(L)<0$, and $\phi\neq 0$. Then $\phi|_D$ is nilpotent with kernel containing the maximal destabilizing bundle of $E|_D$.
\end{lemma}
\begin{proof}
    Note that for any $\lambda\in \Div$ and $s \in H^0(X,\lambda^2)$, $s$ is a perfect square as squaring gives an inclusion $H^0(X,\lambda)\hookrightarrow H^0(X,\lambda^2)$ and $h^0(X,\lambda^2)\leq 1$. If $(E,\phi)$ are as in the statement of the Lemma, then the fact that $\det(\phi)$ is a square implies that either $\phi$ is nilpotent or its spectral cover has two irreducible components. In the first case, $\ker(\phi)$ contains an invariant sub-line bundle $M$ of $\phi$. Using the exact sequence $$\begin{tikzcd}
    0 \arrow[r] &\Hom(M,E')\arrow[r] & \Hom(M,E)\arrow[r] & \Hom(M|_D, L)=0,
    \end{tikzcd}$$
    we see that the inclusion of $M$ factors through the elementary modification. This implies that $\phi|_D$ is nilpotent with kernel $E|_D/L$.

    If $\phi$ is not nilpotent, then there are two $\phi$-invariant sub-line bundles $M$ and $M'$ which generically span $E$. Again, the inclusions of $M$ and $M'$ factor through the elementary modification, so $M|_D$ and $M'|_D$ are both subsheaves of $E|_D/L$. Since $\phi$ is uniquely determined by its values away from $D$, we must have $E|_D/L$ contained in the kernel of $\phi|_D$.
\end{proof}
\begin{proposition}
    Let $E$ be a rank-2 vector bundle on a Hopf surface $X$ with algebraic dimension 0, and suppose that $E$ has an allowable elementary modification \begin{equation}\label{6.10-1}\begin{tikzcd}
        0\arrow[r] & E'\arrow{r}{\iota} & E \arrow[r] & j_*L\arrow[r] & 0
    \end{tikzcd}\end{equation}
    for $j:D\to X$ the inclusion of a prime divisor. 
    Let \begin{equation}\label{6.10-2}\begin{tikzcd}
        0 \arrow[r] & E\otimes \mc{O}_X(-D) \arrow{r}{\iota'} & E' \arrow[r] & j_*(L')\arrow[r] & 0,
    \end{tikzcd}\end{equation}
    be the complementary elementary modification. The map from $\sEnd(E')$ to $\sEnd(E)\otimes \mc{O}_X(D)$ given by $\phi\mapsto \iota\circ \phi\circ \iota'$ induces an isomorphism between $H^0(X,\sEnd_0(E')\otimes \lambda)$ and $H^0(X,\sEnd_0(E)\otimes \lambda\otimes \mc{O}_X(D))$ for any $\lambda \in \Div$.
\end{proposition}
\begin{proof}
    First consider the exact sequence $$\begin{tikzcd}
    H^0(X,\sHom(E',E)\otimes \lambda)\arrow{r}{\iota'^*} & H^0(X,\sEnd(E)\otimes \lambda)\otimes \mc{O}_X(D)) \arrow{r}{p_D} & \Ext^1(j_*L', E\otimes \lambda)
\end{tikzcd}$$ obtained by applying $H^0(X,\sHom(\cdot, E\otimes \lambda)$ to the exact sequence \eqref{6.10-2}. As in \cite[p.42]{Friedman98}, $$\Ext^1(j_*L', E\otimes \lambda)=H^0(D,E|_D\otimes L'^{-1})\simeq\C.$$ The map $p_D$ gives the component of $\phi|_D$ which maps $L'$ to $L'$, which is zero for any trace-free Higgs field by Lemma \ref{alg-zero-nilpotent}. Thus every Higgs field $\phi\in H^0(X,\sEnd_0(E)\otimes \lambda\otimes \mc{O}_X(D))$ is of the form $\psi\circ\iota'$ for some $\psi \in H^0(X,\sHom(E,E')\otimes \lambda)$.

Now consider the exact sequence $$\begin{tikzcd}
    0 \arrow[r] & H^0(X,\sEnd(E')\otimes \lambda)\arrow[r] & H^0(X,\sHom(E,E')\otimes \lambda)\arrow[r] & H^0(D, E'^\vee|_D \otimes L)
\end{tikzcd}$$ obtained by applying $H^0(X,\sHom(\cdot, E\otimes \lambda))$ to the exact sequence \eqref{6.10-1}. Suppose there is a $\phi =\psi\circ \iota' \in H^0(X,\sEnd(E)\otimes \lambda\otimes \mc{O}_X(D))$ so that $\psi$ is not of the form $\iota\circ\phi_0$ for some $\phi_0 \in H^0(X,\sEnd_0(E)\otimes \lambda)$. Then $\psi$ is such that $p_L\circ \psi|_D\neq 0$, where $p_L$ is the projection from $E|_D$ to $L$. However, one can check that then $p_L\circ (\psi\circ\iota')|_D\neq 0$. By Lemma \ref{alg-zero-nilpotent}, we must have that $\psi\circ \iota'$ is a multiple of the identity. Therefore, any $\phi \in H^0(X,\sEnd_0(E)\otimes \lambda\otimes \mc{O}_X(D))$ is of the form $\iota\circ\phi_0\circ \iota'$ for some $\phi_0\in H^0(X,\sEnd_0(E')\otimes \lambda)$.
\end{proof}

Suppose that $E$ is a vector bundle satisfying the hypotheses of Proposition \ref{elementary-modification-sequence}, so that there is a sequence of allowable elementary modifications 

\section{Symplectic foliations and Poisson subspaces}\label{foliations}
Let $(E,\Phi)$ be a non-zero rank-2 trace-free co-Higgs bundle over a Hopf surface $X$, and let $(\mb{P}(E), \sigma_\Phi)$ be the associated Poisson manifold. 
The goal of this section is to describe the symplectic foliation associated to $\sigma_\Phi$. 
Since $\mb{P}(E)$ is three-dimensional and symplectic leaves have even dimension, it suffices to consider symplectic leaves of dimensions zero and two.
\subsection{Zero-dimensional symplectic leaves}
As shown in \cite[Theorem 4.1.3]{AliMedinaThesis}, the union of the zero-dimensional symplectic leaves of $\sigma_\Phi$ is precisely the \emph{eigen-variety} of $\Phi$.
\begin{definition}
    If $(E,\Phi)$ is any twisted Higgs bundle, the \emph{eigen-variety} $\mc{E}(\Phi)\subset \mb{P}(E)$ of $\Phi$ is the set $$\mc{E}(\Phi):=\left\{(x,[v]) \in \mb{P}(E): \Phi_x(v)\wedge v=0\in \wedge^2E_x\otimes \mc{T}_{X,x}\right\}.$$
\end{definition}
One particular case where we can describe an irreducible component of the eigen-variety is when the Higgs field $\Phi$ has an invariant sub-line bundle.
\begin{proposition}
    If $\Phi\neq 0$ and $\iota:L\hookrightarrow E$ is an invariant sub-line bundle of $\Phi$ with torsion-free co-kernel, then $\mc{E}(\Phi)$ will contain the vanishing locus $V(L)\subset \mb{P}(E)$ as an irreducible component, where $V(L):=(\mathrm{Proj}_X(\mathrm{Sym}(\im(\iota^\vee)))$. 
\end{proposition}
\begin{proof}
    Recall that sections of a projective bundle $\mb{P}(E)$ are in one-to-one correspondence with surjective maps $f:E^\vee\to \lambda$, where $\lambda\in \Pic(X)$ is a line bundle. The corresponding section is the one induced by the sub-line bundle $f^\vee:\lambda^\vee\to E$. Thus, if our invariant sub-line bundle $\iota:L\to E$ is such that $\iota^\vee$ is surjective, the eigen-variety $\mc{E}(\Phi)$ will necessarily include the section $\mb{P}(L)=\mathrm{Proj}_X(\mathrm{Sym}(L^\vee))$.

    If instead $\iota^\vee$ is not surjective, there is a codimension-two analytic subspace $Z \subset X$ so that $\im(\iota^\vee)\simeq L^\vee\otimes I_Z$. Above the complement $U:=X\setminus Z$ we have a section of $\mb{P}(E|_U)$ by the same argument as in the previous paragraph, so since the eigen-variety $\mc{E}(\Phi)$  is closed it must contain the closure of this section in $\mb{P}(E)$. If we set $s:\mathrm{Bl}_Z(X)\to X$ to be the blow-up map of $X$ centred at $Z$, then there is a natural surjective map $s^*E\twoheadrightarrow I_Z\cdot \mc{O}_{\mathrm{Bl}_Z(X)}\otimes L^\vee$ with $I_Z\cdot \mc{O}_{\mathrm{Bl}_Z(X)}\otimes L^\vee\simeq L^\vee\otimes \mc{O}_{\mathrm{Bl}_Z(X)}(-D_Z)$ by the universal property of the blow-up, where $D_Z$ is the exceptional divisor. This map gives an inclusion of $\mathrm{Proj}_X(\mathrm{Sym}(L^\vee\otimes I_Z)$ into $\mb{P}(E)$ whose restriction to $U$ is the same as the section described by $\iota^\vee|_U$.
\end{proof}
If $(E,\Phi)$ is a non-zero rank-2 trace-free co-Higgs bundle that admits an invariant sub-line bundle, take $L_1$ and $L_2$ to be two invariant sub-line bundles with torsion-free co-kernel (with $L_1=L_2$ if $\Phi$ is everywhere nilpotent). We can write $\Phi=\phi_0\otimes \xi$ as in Proposition \ref{factor-Higgs}, and the eigen-variety is given by the union $$\mc{E}(\Phi)=V(L_1)\cup V(L_2)\cup p^{-1}(V(\xi))\subseteq \mb{P}(E)$$
where $p:\mb{P}(E)\to X$ is the bundle map.

\subsection{Two-dimensional symplectic leaves}
As shown in Lemma \ref{factor-Higgs}, for any rank-2 trace-free co-Higgs bundle $(E,\Phi)$ on a compact complex manifold $X$ there is a line bundle $\lambda\in \Pic(X)$ and $\phi_0\in H^0(X,\sEnd_0(E)\otimes\lambda)$ and $\xi\in H^0(X,\mc{T}_X\otimes \lambda^{-1})$ so that $\Phi=\phi_0\otimes \xi$. On the complement $U=\mb{P}(E)\setminus \mc{E}(\Phi)$ of the eigen-variety of $\Phi$, we can write the Poisson structure $\sigma_\Phi$ as $\sigma_\Phi|_U=\psi\wedge p^*\zeta$, where $\psi\in H^0(U,\mc{T}_p)$, $\zeta\in H^0(p(U),\mc{T}_X)$, and both $\psi$ and $\zeta$ are nowhere-vanishing. For any integral curve $\gamma$ of $\zeta$, there is a symplectic leaf of $\sigma_\Phi$ of the form $p^{-1}(\gamma)\cap U$. Note that for any choice of $\zeta$ such that $\sigma_\Phi|_U=\psi\wedge p^*\zeta$, there will be a nowhere-vanishing holomorphic function $f \in \mc{O}_X(p(U))$ so that $\zeta=f\xi$, so the integral curves $\gamma$ are uniquely determined by $\xi$. In particular, if $\pi_*\xi$ vanishes at $x \in B\setminus \pi(V(\xi))$, then $p^{-1}(\pi^{-1}(x))\cap U$ will be an embedded symplectic leaf of $(\mb{P}(E),\sigma_\Phi)$.

\begin{proposition}
    Let $(E,\Phi)$ be a rank-2 co-Higgs bundle on a non-Kähler elliptic surface $\pi:X\to B$, and let $(\mb{P}(E),\sigma_\Phi)$ be the corresponding holomorphic Poisson manifold. Write $\Phi=\phi_0\otimes \xi$ with $\phi_0\in H^0(\sEnd_0(E)\otimes \lambda^{-1})$, $\xi\in H^0(\mc{T}_X\otimes \lambda)$ as in Lemma \ref{factor-Higgs}. If $M\subset \mb{P}(E)$ is a closed Poisson subspace which is reduced and irreducible, then  $M\subseteq \mc{E}(\Phi)$, $M=p^{-1}(\pi^{-1}(x))$ for some $x \in B$ with $\pi_*\xi(x)=0$, or $M=\mb{P}(E)$.
\end{proposition}
\begin{proof}
    Recall that a closed reduced subspace of a holomorphic Poisson manifold is a Poisson subspace if and only if it is a union of symplectic leaves. Therefore the previous descriptions of symplectic leaves for $(\mb{P}(E),\sigma_\Phi)$ can be used to classify the closed Poisson subspaces which are reduced and irreducible. 
    
    First suppose that $M$ contains no two-dimensional symplectic leaves. Since $\mc{E}(\Phi)$ is equal to the union of all zero-dimensional symplectic leaves, $M\subset \mc{E}(\Phi).$

    If $M$ contains a unique two-dimensional symplectic leaf, then $M$ will be equal to the closure of this leaf $F$. Furthermore, $M$ will have dimension equal to two as $M\subset F\cup \mc{E}(\Phi)$ and $\dim(F\cup\mc{E}(\Phi))=2$. The symplectic leaf $F$ corresponds to some integral curve $\gamma$ of $\xi$. Since $\dim\overline{\gamma}=\dim\overline{F}-1=1$, $\overline{\gamma}$ must be a prime divisor of $X$. Every prime divisor of $X$ is of the form $\pi^{-1}(x)$ for some $x \in B$, and since $\gamma\subset \pi^{-1}(x)$ is an integral curve of $\xi$, $\pi_*\xi(x)=0$.

    Finally, suppose that $M$ contains more than one two-dimensional symplectic leaf. Since we have assumed that $M$ is irreducible, this implies that $\dim M=3$. The unique three-dimensional closed subspace of $\mb{P}(E)$ is $\mb{P}(E)$ itself, so $M=\mb{P}(E)$.
\end{proof}

For the case of a co-Higgs bundle $(E,\Phi)$ on a Hopf surface, the description of the symplectic leaves is somewhat more explicit. Using Lemma \ref{factor-Higgs}, we can write $\Phi$ in the form $\Phi=\phi_0\otimes \xi$, where $\phi_0$ is a $\lambda$-twisted Higgs field over $E$ for some $\lambda \in \Div$ and $\xi \in H^0(X,\mc{T}_X\otimes \lambda^{-1})$. The Poisson structure $\sigma_\Phi$ lifts to a Poisson structure $f\frac{\partial}{\partial t} \wedge(\alpha\frac{\partial}{\partial x}+\beta\frac{\partial}{\partial y})$ on $(\C^2\setminus \{0\})\times \mb{P}^1$ where $\alpha$ and $\beta$ are interpreted as homogeneous polynomials in $\C[x,y]$ of the appropriate degree, $t$ is a coordinate in the $\mb{P}^1$-fibre direction, and $f$ vanishes exactly on the pre-image of the eigen-variety $\mc{E}(\Phi)$. Note that while $\alpha\frac{\partial}{\partial x}+\beta\frac{\partial}{\partial y}$ does not necessarily descend to a well-defined vector field on $X$, it does descend to a projectivized vector field in $\mb{P}(\mc{T}_X)$ since it is the lift of a section of $\mc{T}_X\otimes \lambda^{-1}$ and $\lambda$ is defined by a constant factor of automorphy, meaning that for any non-constant integral curve $\gamma$ of $\xi$ there is a symplectic leaf of $(\mb{P}(E), \sigma_\Phi)$ of the form $p^{-1}(u(\gamma))\cap (\mb{P}(E)\setminus \mc{E}(\Phi)),$ where $u:\C^2\setminus\{0\}\to X$ is the universal covering map. When $X$ is elliptic so that there are integers $r_2\geq r_1\geq 1$ so that $a^{r_1}=b^{r_2}$, the map $d\pi:\mc{T}_X\otimes \lambda^{-1}\to \pi_*(\pi^*\mc{T}_{\mb{P}^1}\otimes \lambda^{-1})\simeq \pi_*\lambda^{-1}\otimes \mc{O}_{\mb{P}^1}(2)$ is given by $d\pi(\alpha\frac{\partial}{\partial x}+\beta\frac{\partial}{\partial y})\mapsto \pi_*(r_1\alpha x^{r_1-1}y^{r_2}-r_2\beta x^{r_1}y^{r_2-1})$, for any point $q=[s:t]=[x^{r_1}:y^{r_2}] \in \mb{P}^1$ satisfying $r_1t\pi_*(\alpha x^{r_1-1})(q)-r_2s\pi_*(\beta y^{r_2-1})(q)=0$ and $\Phi|_{\pi^{-1}(q)}\neq 0$, there is a symplectic leaf of the form $p^{-1}\left(\pi^{-1}\left(q\right)\right)\cap \left(\mb{P}(E)\setminus \mc{E}(\Phi)\right)$. Since every closed curve of $X$ is a fibre of $\pi$, the leaves of this type are the only 2-dimensional symplectic leaves of $\sigma_\Phi$ that admit a locally closed embedding into $\mb{P}(E)$.

When $X$ is generic, there will be a two-dimensional leaf contained in $p^{-1}(V(x))$ when $\alpha(0,y)=0$ and one contained in $p^{-1}(V(y))$ when $\beta(x,0)=0$. Similarly, when $X$ is exceptional there will be a two-dimensional symplectic leaf contained in the unique divisor $V(y)$ if $\beta(x,0)=0$.

In principle, the 2-dimensional symplectic leaves can be directly computed by finding the integral curves of the vector field $\alpha\frac{\partial}{\partial x}+\beta\frac{\partial}{\partial y}$ on $\C^*\setminus \{0\}$. However in practice these computations become intractable to do explicitly unless the vector field is is of sufficiently low degree.

\begin{example}
    Suppose that $X$ is an elliptic Hopf surface with $a^{r_1}=b^{r_2}$ for some $r_2\geq r_1\geq 1$ and the co-Higgs bundle $(E,\Phi)$ on $X$ is non-zero nilpotent with $\ker(\Phi)=L$. Then for $\lambda:=\det(E)\otimes L^{-2}\in \Div$ there are $\phi_0\in H^0(X,\sEnd_0(E)\otimes \lambda)$, $\alpha \in H^0(X,L_a\otimes\lambda^{-1})$, and $\beta \in H^0(X,L_b\otimes \lambda^{-1})$ so that $\phi_0$ is nowhere vanishing and $\Phi=(\alpha\phi_0, \beta\phi_0)$. Since $\Phi$ is nilpotent, the eigen-variety is given by $\mc{E}(\Phi)=V(L)\cup p^{-1}V(\alpha,\beta)$. The embedded 2-dimensional symplectic leaves are all of the form $p^{-1}(\pi^{-1}(q))\cap (\mb{P}(E)\setminus \mc{E}(\Phi))$ for some $q=[s:t]=[x^{r_1}:y^{r_2}]\in \mb{P}^1$ with $r_1t\pi_*(\alpha x^{r_1-1})(q)-r_2s\pi_*(\beta y^{r_2-1})(q)=0$ and $\Phi|_{\pi^{-1}(q)}\neq 0$. If $\deg(\lambda)<0$ and there is a section $t \in H^0(X,\lambda^{-1})$ so that $\alpha=xt$ and $\beta=yt$, then every symplectic leaf is contained in a fibre of $\pi\circ p$.
\end{example}

\begin{example}
    Suppose that $E$ is a rank-2 vector bundle such that, after applying an allowable elementary modification along some irreducible divisor $D$, the resulting bundle $E_0$ is of the form $E_0\simeq L\oplus L^{-1}$ where $L^2 \not\in \Div$. Every non-zero co-Higgs bundle with underlying vector bundle $E$ is of the form $\Phi=\phi_0\otimes \xi$ with $\xi \in H^0(X,\mc{T}_X\otimes \mc{O}_X(-D))\setminus \{0\}$, and $\phi_0$ is the unique-up-to-scaling section in $H^0(X,\sEnd_0(E)\otimes \mc{O}_{X}(D))$. In this case the eigen-variety is given by $\mc{E}(\Phi)=V(L)\cup V(L^{-1})\cup p^{-1}(V(\xi))$. 
    
    If $X$ is classical, then there are constants $\alpha,\beta \in \C$ so that $\xi=\alpha\frac{\partial}{\partial x}+\beta\frac{\partial}{\partial y}$. In this case $V(\xi)=\emptyset$, and the 2-dimensional symplectic leaves are all of the form $p^{-1}(u(V(\alpha x+\beta y-c)))\cap (\mb{P}(E)\setminus \mc{E}(\Phi))$ for some $c \in \C$.

    If $X$ is resonant and $D\neq V(y)$, then $\xi=\alpha\frac{\partial}{\partial x}$ for some $\alpha \in \C$, so that $V(\xi)=\emptyset$ and the 2-dimensional symplectic leaves are of the form $p^{-1}(u(V(x-c)))\cap (\mb{P}(E)\setminus \mc{E}(\Phi))$ for some $c \in \C$. If instead $D=V(y)$, then $\xi=\alpha y^{r-1}\frac{\partial}{\partial x}+\beta\frac{\partial}{\partial y}$ for some $\alpha,\beta\in \C$. When $\beta\neq 0$, $V(\xi)=\emptyset$ and the 2-dimensional leaves are all of the form $p^{-1}(u(V(rx-\alpha y^r-c)))\cap (\mb{P}(E)\setminus \mc{E}(\Phi))$ for some $c \in \C$. If $\beta=0$ and $\alpha\neq 0$, then $V(\xi)=V(y)$ and the 2-dimensional leaves are of the form $p^{-1}(u(V(y-c)))\cap (\mb{P}(E)\setminus \mc{E}(\Phi))$ for some $c \in \C^*$.

    If $X$ is hyper-resonant or generic, then the assumption that $\xi\neq 0$ implies that either $D=V(x)$ or $D=V(y)$. In either case, $V(\xi)=\emptyset,$ and the two-dimensional symplectic leaves are of the form $p^{-1}(u(V(x-c)))\cap (\mb{P}(E)\setminus \mc{E}(\Phi))$ for some $c \in \C$ (when $D=v(x)$) or $p^{-1}(u(V(y-c)))\cap (\mb{P}(E)\setminus \mc{E}(\Phi))$ for some $c \in \C$ (when $D=V(y)$).

    If $X$ is exceptional, the assumption that $\xi\neq 0$ forces $\xi=\alpha y^{r-1}\frac{\partial}{\partial x}$ so that $V(\xi)=V(y)$ and the 2-dimensional symplectic leaves are all of the form $p^{-1}(u(V(y-c)))\cap (\mb{P}(E)\setminus \mc{E}(\Phi))$ for some $c \in \C^*$.
\end{example}

\section*{Acknowledgements}
The author was supported by an NSERC postdoctoral fellowship [PDF - 587400 - 2024] in the writing of this article, as well as by Steven Rayan and the University of Saskatchewan. 
The author also wishes to thank Ruxandra Moraru and Raphaël Belliard for helpful discussions relating to this research.
\bibliographystyle{alpha}
\bibliography{references}
\end{document}